\documentclass[reqno]{amsart}

\usepackage{amsfonts}
\usepackage{amssymb}
\usepackage{amsmath}
\usepackage{mathrsfs}
\usepackage{amsthm}
\usepackage[colorlinks,linkcolor=blue,citecolor=blue]{hyperref}

 \newtheorem{theorem}{Theorem}[section]
\newtheorem{corollary}[theorem]{Corollary}
\newtheorem{lemma}[theorem]{Lemma}
\newtheorem{proposition}[theorem]{Proposition}

\theoremstyle{remark}
\newtheorem{remark}[theorem]{Remark}

\theoremstyle{definition}

\numberwithin{equation}{section}

\begin{document}

\title[The anisotropic $N$-Liouville equation]{Classification of solutions to the anisotropic $N$-Liouville equation in $\mathbb{R}^N$}

\author[G. Ciraolo]{Giulio Ciraolo}
\address[G. Ciraolo]{Department of Mathematics ``Federigo Enriques"\\
Universit\`a degli Studi di Milano\\ Via Cesare Saldini 50, 20133 Milan\\
Italy}
\email{giulio.ciraolo@unimi.it}

\author[X. Li]{Xiaoliang Li}
\address[X. Li]{Department of Mathematics ``Federigo Enriques"\\
	Universit\`a degli Studi di Milano\\ Via Cesare Saldini 50, 20133 Milan\\
	Italy}
\email{xiaoliang.li@unimi.it}

\thanks{}
\subjclass{35J92, 35B06, 35B08, 35B40, 35B53}

\keywords{Liouville equation; Quasilinear anisotropic elliptic equations; Finsler $N$-Laplacian; Logarithmic singularity}

\begin{abstract}
Given $N\geq 2$, we completely classify solutions to the anisotropic $N$-Liouville equation $$-\Delta_N^H\,u=e^u \quad\text{in }\mathbb{R}^N,$$ under the finite mass condition $\int_{\mathbb{R}^N} e^u\,dx<+\infty$. Here $\Delta_N^H$ is the so-called Finsler $N$-Laplacian induced by a positively homogeneous function $H$. As a consequence in the planar case $N=2$, we give an affirmative answer to a conjecture made in [G. Wang and C. Xia, J. Differential Equations 252 (2012) 1668--1700].
\end{abstract}

\maketitle

\section{Introduction}
A classical problem in the study of qualitative properties of elliptic PDEs is the classification of solutions to the Liouville equation
\begin{equation} \label{eq_L}
-\Delta u = e^u \quad \textmd{in } \mathbb{R}^2.
\end{equation}
This study was initiated by Liouville \cite{Liouville1853}, who found that solutions to \eqref{eq_L} can be represented by analytic functions on the complex plane $\mathbb{C}$. Several years later, Chen and Li \cite{CL1991} proved by the method of moving planes that, under the finite mass condition $$\int_{\mathbb{R}^2} e^u\,dx<+\infty,$$ solutions to \eqref{eq_L} can be completely classified and they are given by a family of radially symmetric functions
$$
u(x)=\log \frac{8 \lambda^2}{(1+\lambda^2 |x-x_0|^2)^2} 
$$
for all $\lambda>0$ and $x_0 \in\mathbb R^2$.

This result has been subsequently revisited by using different techniques (see \cite{CW1994,CK1995}), and recently, its natural generalization to higher dimensions has also been tackled in \cite{Esposito2018}, where the author classified all finite-mass solutions to the equation
\begin{equation} \label{eq_LN}
-\Delta_N u = e^u \quad \textmd{in } \mathbb{R}^N.
\end{equation}
 Here, $\Delta_Nu=\mathrm{div}\,(|\nabla u|^{N-2}\nabla u)$ denotes the $N$-Laplacian operator and $N\geq 2$ is an integer. We will come back to these results later in the Introduction.

From the variational point of view, it is known that the operator $\Delta_N$ arises as first variation of the functional 
$$
\int |\nabla u|^N\,dx \,,
$$
which is conformally invariant \cite{Loe}.
This suggests that, when investigating qualitative properties of solutions to problems of type \eqref{eq_LN}, various techniques which exploit the isotropy of $\mathbb R^N$ endowed with the Euclidean norm $|\cdot|$ and/or the conformal invariance of $\Delta_N$ can be used, such as the method of moving planes/spheres, Kelvin transform and so on; see \cite{CL1991,CL1993,CK1994,PT2001,Esposito2018} for instance. These techniques, however, seem to be not suitable whenever the problems are considered in a possibly anisotropic medium. Here the term ``anisotropic" stands for the non-Euclidean way in which the ambient space $\mathbb R^N$ is measured, and it is encoded into a norm $H$ (possibly \emph{asymmetric})\footnote{As precisely explained right after here, a norm we say in this paper always refers to a positively homogeneous and convex function which is not necessarily centrally symmetric, i.e. $H(\xi)\neq H(-\xi)$ in general. This can be viewed as a generalization of the usual concept of norm.}, i.e. $H$ is a convex function on $\mathbb{R}^N$ which is \emph{positively} homogeneous of degree one and satisfies
$$
H(\xi)>0\quad\text{for any }\xi\in\mathbb{R}^{N}\setminus\{0\}.
$$ 
We notice that $\mathbb{R}^N$ endowed with the norm $H$ can be viewed as a Finsler manifold, which reflects into a rich geometric structure to explore since it breaks the usual symmetry properties coming from the peculiarities of the Euclidean norm, like the rotational invariance.

Thus, it is natural to understand whether the Finsler setting preserves some kind of symmetry to be characterized and what kind of behaviors may result for solutions of equations in this setting. More precisely, in the case of Liouville equation \eqref{eq_LN} it is interesting to consider by generalizing the operator to the \emph{Finsler $N$-Laplacian} $\Delta_N^H$, which is given below by \eqref{eq:def-F-p-Lap} with $p=N$ and naturally appears in the Euler--Lagrange equation for the Wulff-type functional $$\int H^N(\nabla u)\,dx$$
where the gradient of a function $u$ is measured by the norm $H$.

Such a generalization to anisotropic (or Finsler) PDEs is also of interest in its own right, as it is motivated by concrete applications in many fields, such as in the study of capacity and sharp inequalities of geometric-functional type, in blowup analysis, in digital image processing, crystalline mean curvature flow and crystalline fracture theory. These studies have attracted a growing attention in the past decades; see for instance \cite{BC2018,CFR2020,CRS2016,FMP2010,FM1991,Taylor1978,DP1992,BFK2003,WX2012,ZZ2019} and the references therein.

 However, this generalization is sometimes challenging, since one has to abandon typical arguments which apply only in the Euclidean setting and find new ways to handle. As an example, in the recent paper \cite{CFR2020} the first author, Figalli and Roncoroni studied the critical anisotropic $p$-Laplace equation 
\begin{equation}\label{in-eq:critical}
	-\Delta_p^H u=u^{\frac{Np}{N-p}-1}\quad\text{in } \mathbb{R}^N,
\end{equation}
where $1<p<N$ and $\Delta_p^H$ is the \emph{Finsler $p$-Laplacian} induced by $H$ 
\begin{equation}\label{eq:def-F-p-Lap}
	\Delta_p^H u=\mathrm{div}\left(H^{p-1}(\nabla u)\nabla H(\nabla u)\right).
\end{equation}
 Under the assumption that
\begin{equation}\label{in-eq:H2}
	H^2 \text{ is uniformly convex}\,\footnote{ That is, $\exists\, \lambda>0$, s.t. $D^2(H^2)(\xi)\ge \lambda\,\mathrm{Id}\quad \forall\xi\in\mathbb{R}^N\setminus\{0\}$.},
\end{equation}
the authors in \cite{CFR2020} classified all positive {$\mathcal{D}^{1,p}$}-solutions to \eqref{in-eq:critical},
extending previous works \cite{CGS1989,CL1991,DMMS2014,Vet2016,Sci2016} dealing with \eqref{in-eq:critical} in the Euclidean case, i.e. when $H$ is the Euclidean norm $|\cdot|$. 

In the present paper, we are interested in the anisotropic $N$-Liouville equation, which is seen as a counterpart of \eqref{in-eq:critical} in the case $p=N$ as well as the anisotropic counterpart of \eqref{eq_LN}. More precisely, we consider the following problem
\begin{equation}\label{in-eq:Lio}
\begin{cases}
-\Delta_N^H \,u=e^u &\text{in }\mathbb{R}^N\\
\int_{\mathbb{R}^N}e^u\,dx<+\infty.
\end{cases}
\end{equation}
This model stems from the blow-up analysis of problems with exponential nonlinearities on bounded domains, and it describes the limiting profile of non-compact solution-sequences near a blow-up point; see for instance \cite{WX2012,XG2016,ZZ2019,BM1991,LS1994,AP1997,EM2015}.

The aim of this paper is to address the classification of solutions to \eqref{in-eq:Lio}. Here a solution of  \eqref{in-eq:Lio} stands for a function $u\in W_{\mathrm{loc}}^{1,N}(\mathbb{R}^N)$ such that $e^u\in L^1(\mathbb{R}^N)$ and $$\int_{\mathbb{R}^N}H^{N-1}(\nabla u)\left<\nabla H(\nabla u),\nabla\varphi\right>dx=\int_{\mathbb{R}^N}e^u\varphi\,dx$$
for all $\varphi\in W_0^{1,N}(\Omega)\cap L^\infty(\Omega)$ with $\Omega\subset\mathbb{R}^N$ bounded.  
 
Assuming \eqref{in-eq:H2}, we are able to provide a complete classification result for problem \eqref{in-eq:Lio}. More precisely, we prove that there is a family of explicit logarithmic functions exhausting all possible solutions to \eqref{in-eq:Lio}. In particular, due to scaling and translation invariance, when $N=2$ this affirmatively answers a conjecture proposed by Wang--Xia \cite{WX2012} on the uniqueness of solutions to \eqref{in-eq:Lio} in $\mathbb{R}^2$. 

To be more specific, we set
\begin{equation}\label{in-eq:def-hat-H0}
	\hat{H}_0(x)=H_0(-x)
\end{equation}
for $x\in\mathbb{R}^N$, where $H_0$ is the dual function of $H$ defined by
\begin{equation*}
H_0(x):=\sup_{\xi\in\mathbb{S}^{N-1}}\frac{\left<x,\xi\right>}{H(\xi)}.
\end{equation*}
Our main result is the following.

\begin{theorem}\label{thm:main}
Let $u$ be a solution of \eqref{in-eq:Lio} and assume that $H\in C^2(\mathbb{R}^N\setminus\{0\})$ satisfies \eqref{in-eq:H2}. Then 
\begin{equation}\label{in-eq:u}
u(x)=\log\frac{N\left(\frac{N^2}{N-1}\right)^{N-1}\lambda^N}{\left[1+\lambda^{\frac{N}{N-1}}\hat{H}_0^{\frac{N}{N-1}}(x-x_0)\right]^N}
\end{equation}
for some $\lambda>0$ and $x_0\in\mathbb{R}^N$.
\end{theorem}

We notice that Theorem \ref{thm:main} is already available in literature when $H$ is the Euclidean norm, as mentioned at the beginning of the Introduction. Indeed, if $H(\xi)=|\xi|$ and $N=2$ then the equation in \eqref{in-eq:Lio} reduces to \eqref{eq_L} and Theorem \ref{thm:main} for \eqref{eq_L} was first proved by Chen--Li \cite{CL1991} via the method of moving planes. Other two proofs in this semilinear case were given later by Chou--Wan \cite{CW1994} through complex analysis, and by Chanillo--Kiessling \cite{CK1995} by using a Pohozaev identity combined with the isoperimetric inequality, respectively. In the general case of the $N$-Laplacian, with $N\geq2$, problem \eqref{in-eq:Lio} was recently treated by Esposito \cite{Esposito2018}, who completely proved Theorem \ref{thm:main} in the Euclidean case by exploiting the Kelvin transform.

 We also point out that, when $N=2$ and $H$ is a centrally symmetric norm over $\mathbb{R}^2$, the result in Theorem \ref{thm:main} was conjectured in  \cite[Conjecture 6.1]{WX2012}.

Theorem \ref{thm:main} extends the classification result in \cite{Esposito2018} into a Finsler setting. It was reasonable to expect that this extension is not trivial, as suggested by the conjecture mentioned above, and as might be also foreseen from the discussions in \cite{CFR2020} for critical equations \eqref{in-eq:critical} as well as from other related studies in this spirit (see for instance \cite{BC2018,CS2009,CS2023,FMP2010,CRS2016,WX2011}).

Indeed, for problem \eqref{in-eq:Lio} with $H$ possibly non-Euclidean, it would be quite a delicate issue to derive the asymptotic behavior of the solutions at infinity, as previous techniques used in relevant literature are not applicable. Specifically, in order to gain the desired asymptotic property of solutions to \eqref{in-eq:Lio} when $H(\xi)=|\xi|$, the authors in \cite{Esposito2018} and \cite{CL1991,CK1995} made use of the Kelvin transform and the Green representation formula for \eqref{eq_L},  respectively. However, \emph{neither} of them are available in our setting. Regarding \cite{CFR2020}, the derivation of decay estimates for solutions of \eqref{in-eq:critical} followed the idea in V\'etois \cite{Vet2016}, using a scaling argument together with a doubling property from \cite{PPS2007}. Such an approach was originally developed by Pol\'a\v{c}ik--Quittner--Souplet \cite{PPS2007,PPS2007p} to obtain power-law singularity estimates for problems with subcritical Sobolev growth, but it does not seem to be effective for Liouville type equations whose solutions with finite mass have logarithmic singularities instead (see \cite{Esposito2021}). Therefore, in order to tackle that issue, we were led to seek a general and non-conformal way of determining the asymptotics at infinity for solutions to problems of type \eqref{in-eq:Lio}, which applies to the anisotropic framework.

\subsection{Derivation of asymptotic behavior at infinity}\label{sec:derivation}
Let us first go back to the idea of Esposito \cite{Esposito2018} dealing with \eqref{in-eq:Lio} in the Euclidean case (see also \cite{Esposito2021}) in order to show a contrast with the approach we shall develop here. 

In \cite{Esposito2018,Esposito2021}, thanks to the conformal invariance of the $N$-Laplacian, Esposito performed the inversion $\hat{u}$ of a solution $u$ to \eqref{eq_LN} (i.e. $\hat{u}(x)=u(\frac{x}{|x|^2})$), and deduced the asymptotic behavior of $u$ at infinity by analyzing the singularity of $\hat{u}$ at the origin. The crucial point there consists in exploiting the equation satisfied by $\hat{u}$ and showing that, on a punctured ball $B_r\setminus\{0\}$, $\hat{u}$ can be decomposed into two parts:
\begin{equation}\label{in-eq:decom-hat-u}
	 \hat{u}=v+w
\end{equation}
where $v$ is a $N$-harmonic function and $w\in L^\infty(B_r)$. Then according to \cite{KV1986}, it was concluded in \cite{Esposito2018,Esposito2021} that $v$ completely characterizes the logarithmic behavior of $\hat{u}$ at $0$, given by
\begin{equation}\label{in-eq:v-log}
	v-\bar{\gamma}\log |x|\in L^\infty(B_r)
\end{equation}
for some constant $\bar{\gamma}$.

Turning to the Finsler case, a Kelvin-type transform as making the inversion map is not available except for $H$ taking a special form \cite{ERSV2023}. Moreover, it seems hard to derive a logarithmic estimate of type \eqref{in-eq:v-log} for solution $u$ near infinity by carrying out a dual adaptation of the idea \eqref{in-eq:decom-hat-u} in exterior domains. Indeed, the techniques validating \eqref{in-eq:decom-hat-u} in \cite{Esposito2018,Esposito2021} are based on several global estimates of integral type over bounded domains (with bounds depending on the measure of the domain), such as Brezis--Merle type inequality as in Lemma \ref{pre-lem:priori-exp}. In our case, since we have to deal with the solution in a neighborhood of infinity, this approach seems to be not suitable to get appropriate uniform estimates about $u$ near infinity and employ approximation arguments as done in \cite{Esposito2018,Esposito2021} for $\hat{u}$ near the origin.

Instead, here we tackle problem \eqref{in-eq:Lio} by establishing a result on the determination of logarithmic singularity at infinity for exterior solutions to equations of the general form
\begin{equation}\label{in-eq:div-A-f}
	\mathrm{div}\,\mathcal{A}(\nabla u)=f(x)
\end{equation}
where $\mathcal{A}$ and $f$ are given functions satisfying suitable structure conditions. This result is made explicitly in Theorem \ref{sing-thm:log-A-f} and it is based on Serrin's works \cite{Serrin1964,Serrin1965-1,Serrin1965}, which generalizes the corresponding discussions in \cite{Serrin1965} for a class of homogeneous equations. Thanks to Theorem \ref{sing-thm:log-A-f} applied to the case of \eqref{in-eq:Lio}, we are able to characterize the following logarithmic behavior for a solution $u$ of \eqref{in-eq:Lio}
\begin{equation}\label{in-eq:u-log}
	u(x)+\gamma\log\hat{H}_0(x):=G(x)\in L^\infty(\mathbb{R}^N\setminus B_1)
\end{equation}
with $\gamma$ being a constant.

More precisely, \eqref{in-eq:u-log} is validated in the following way. First, from the scaling invariance of \eqref{in-eq:Lio} we obtain that the solution $u$ blows down at infinity with a non-sharp upper bound \eqref{pre-eq:u-upper}. Then an application of Theorem \ref{sing-thm:log-A-f} to \eqref{in-eq:Lio} permits us to find also a preliminary lower estimate for $u$ near infinity. With these, we introduce $\gamma$ as in \eqref{asym-eq:def-gamma} and obtain by exploiting a scaling argument that the function $G$ in \eqref{in-eq:u-log} is bounded from below. In order to show that $G$ is also bounded from above, we construct a family of functions with a uniform behavior of type \eqref{in-eq:u-log} and deduce that they are supersolutions of equation \eqref{in-eq:Lio} outside a large ball. Accordingly, we conclude \eqref{in-eq:u-log} via the comparison principle.

Once \eqref{in-eq:u-log} is shown, we can further derive the asymptotic expansion of $\nabla u$ at infinity in terms of a scaling argument and determine the value of $\gamma$ by the divergence theorem; see Proposition \ref{asym-prop:asym-u-Du}. 


\subsection{Outline of the proof of Theorem \ref{thm:main}}
Let $u$ be a solution to \eqref{in-eq:Lio}. We first determine the precise asymptotic behavior of $u$ at infinity, as well as the one for $\nabla u$. Then we prove Theorem \ref{thm:main} along the same lines as in \cite{CK1995,Esposito2018}, i.e. through a Pohozaev type identity and an isoperimetric argument. More precisely, the proof can be outlined as follows.

Thanks to Proposition \ref{asym-prop:asym-u-Du}, we can apply a Pohozaev type identity derived in Lemma \ref{quant-lem:Pohozaev} to establish a quantization result, which gives the value of $\|e^u\|_{L^1(\mathbb{R}^N)}$ (see Proposition \ref{quan-pro}). This result yields that all the integral inequalities appearing in estimating $\|e^u\|_{L^1(\mathbb{R}^N)}$ from below (i.e. Lemma \ref{pre-lem:eu-lower}) become equalities. In particular, it turns out that each level set of $u$ satisfies the equality case of the anisotropic isoperimetric inequality. Hence, we deduce from the Wulff theorem (see for instance \cite{Wulff1901,Taylor1978,FM1991,FMP2010,CRS2016}) that the level sets of $u$ are Wulff shapes associated with $\hat{H}_0$, as defined in \eqref{in-eq:def-Wulff} below. Furthermore, we demonstrate that these Wulff shapes are concentric (they have the same center). With such symmetry and rigidity, we finally classify the representation \eqref{in-eq:u} for $u$ via an ODE argument.

This approach differs from the one exploited in \cite{CFR2020}, where positive solutions to \eqref{in-eq:critical} are classified by combining several integral identities with a Newton type inequality (see also \cite{BC2018,BNST2008,CS2009,Wein1971} for related techniques applied to Serrin's rigidity problems). Instead, we obtain the classification of solutions to \eqref{in-eq:Lio} from the isoperimetric inequality and our approach is more in the spirit of \cite{Ser2013,DPV2021,BFK2003,KP1994,Lions1981} where symmetry results are proved for bounded domains, as well as of \cite{CL2022,Poggesi2019,XY2021} for exterior overdetermined problems. Moreover, we believe that our approach to derive asymptotic estimates at infinity is a major novelty of the paper and will be useful also in other settings.

\subsection{Notation and organization of the paper}
Throughout the paper, we assume that the function $H$ satisfies \eqref{in-eq:H2}. This assumption enables the Finsler $p$-Laplacian $\Delta_p^H$ with $p>1$ to be a possibly degenerate elliptic operator, such that the equations driven by $\Delta_p^H$ may have the same nature as those quasilinear elliptic equations investigated in \cite{Serrin1964,Serrin1965,Lieberman1988,Tolksdorf1984,DiBen1983}. Actually, according to \cite[Theorem 1.5]{CFV2016}, \eqref{in-eq:H2} is equivalent to that $H$ is \emph{uniformly elliptic}, i.e. the sublevel set $\{\xi\in\mathbb{R}^N:H(\xi)<1\}$ is uniformly convex, as we required  previously in \cite{CL2022}. In this connection, one is referred to \cite[Subsection 2.1]{CL2022} for specific properties of the function $H$ we are considering.

In the whole paper, we set 
\begin{equation}\label{in-eq:def-Wulff}
	B_r^{\hat{H}_0}(x):=\{y\in\mathbb{R}^N:\hat{H}_0(y-x)<r\}
\end{equation}
 and denote by $B_r(x)$ the Euclidean ball with radius $r$ and center $x$; if $x=0$, we would omit it and write $B_r^{\hat{H}_0}$ and $B_r$ for simplicity. In the same way, we introduce the symbol $B_r^{H_0}$ to denote the ball in the metric $H_0$. Here $\hat{H}_0$ and $H_0$ are as in \eqref{in-eq:def-hat-H0}; analogously, we define
 \begin{equation}\label{in-eq:def-hat-H}
 	\hat{H}(\xi)=H(-\xi)\quad\text{for }\xi\in\mathbb{R}^N.
 \end{equation}
 Clearly, $\hat{H}_0$ is the dual function of $\hat{H}$ in a sense as in \eqref{in-eq:def-hat-H0}. In addition, we let $\mathcal{L}^N$ and $\mathcal{H}^{N-1}$ stand for the $N$-dimensional Lebesgue measure and $(N-1)$-dimensional Hausdorff measure, respectively.

The rest of the paper is organized as follows. In Section \ref{sec:pre} we provide two preliminary results concerning \emph{a priori} properties of a solution $u$ to \eqref{in-eq:Lio}. The first shows that $u$ is locally of class $C^{1,\alpha}$ and is globally bounded from above by a logarithmic function. The latter gives a lower \emph{a priori} estimate on $\|e^u\|_{L^1(\mathbb{R}^N)}$. Section \ref{sec:asym} is devoted to deriving the asymptotic behavior at infinity for $u$ and $\nabla u$, which allows us to establish in Section \ref{sec:quant} a quantization result on the mass $\|e^u\|_{L^1(\mathbb{R}^N)}$ by exploiting a Pohozaev type identity. After this, we present the proof of Theorem \ref{thm:main} in Section \ref{sec:pf}. Section \ref{sec:sing} performs analysis on logarithmic singularity at infinity for solutions of the general equation \eqref{in-eq:div-A-f} and also presents several related results for the anisotropic $N$-Laplace equation, which are of independent interest.

\section{Preliminaries}\label{sec:pre}
In this section, we first show that any solution $u$ of \eqref{in-eq:Lio} belongs to $C_{\mathrm{loc}}^{1,\alpha}(\mathbb{R}^N)$ and it is bounded from above by a global logarithmic estimate, see Proposition \ref{pre-pro:global} below. Then, in Lemma \ref{pre-lem:eu-lower} we derive a lower \emph{a priori} estimate on $\|e^u\|_{L^1(\mathbb{R}^N)}$ by applying the anisotropic isoperimetric inequality. These results are inspired by those obtained in \cite{Esposito2018} when $H(\xi)=|\xi|$ and the one established in \cite{Esposito2021} for the generalized $N$-Liouville equations in $B_1\setminus\{0\}$. We mention that our argument for the estimate in Lemma \ref{pre-lem:eu-lower} simplifies the one presented in \cite{Esposito2018}, which was done by approximation; we refer to Remark \ref{pre-rk:non-weight} for further details.

\begin{proposition}\label{pre-pro:global}
Let $u$ be a solution of \eqref{in-eq:Lio}. Then $u^+\in L^\infty(\mathbb{R}^N)$ and $u\in C_{\mathrm{loc}}^{1,\alpha}(\mathbb{R}^N)$ for some $\alpha\in(0,1)$. Moreover, there exists a constant $C>0$ such that 
\begin{equation}\label{pre-eq:u-upper}
u(x)\leq C-N\log |x|\quad\text{for any }x\in\mathbb{R}^N\setminus\{0\}.
\end{equation}
\end{proposition}

To prove this, we need the following two lemmas, which contain fine \emph{a priori} estimates obtained in \cite{Serrin1964,Esposito2018} for $N$-Laplace type equations.

\begin{lemma}\label{pre-lem:priori-exp}
Let $\Omega$ be a bounded domain in $\mathbb{R}^N$ and $f\in L^1(\Omega)$. Let $u\in W^{1,N}(\Omega)$ be a weak solution of 
\begin{equation}\label{pre-eq:ws-f}
	-\Delta_N^H\,u=f\quad\text{in }\Omega
\end{equation}
and let $h\in W^{1,N}(\Omega)$ be the weak solution of
\begin{equation*}
\begin{cases}
\Delta_N^H\,h=0&\text{in }\Omega\\
h=u&\text{on }\partial\Omega.
\end{cases}
\end{equation*}
Then there exists a constant $\mu$, depending only on $N$ and $H$, such that for every $0<\lambda<\mu\|f\|_{L^1(\Omega)}^{-\frac{1}{N-1}}$ it holds
$$\int_{\Omega}e^{\lambda|u-h|}\,dx\leq\frac{\mathcal{L}^N(\Omega)}{1-
\lambda\mu^{-1}\|f\|_{L^1(\Omega)}^{\frac{1}{N-1}}}.$$
\end{lemma}
    
\begin{proof}
	This inequality follows by applying \cite[Proposition 2.1]{Esposito2018} to concrete equation \eqref{pre-eq:ws-f}. 
\end{proof}

We mention that this type of exponential inequality was first obtained by Brezis and Merle in \cite{BM1991} for the Laplacian, and then it was extended to the $N$-Laplacian for $N\geq 2$ (see \cite{AP1997,RW1995} and the references therein). The extension to a Finsler case can be found in \cite{WX2012,XG2016}. More generally, according to what done in \cite{AP1997}, Esposito \cite{Esposito2018} developed such inequality for quasilinear equations of the form $-\mathrm{div}\,\mathbf{a}(x,\nabla u)=f$ which have similar structure as the $N$-Laplace equations.

The next lemma concerns Serrin's local $L^\infty$ estimate in \cite{Serrin1964}.
\begin{lemma}\label{pre-lem:Serrin}
Let $u\in W_{\mathrm{loc}}^{1,N}(\Omega)$ be a weak solution of \eqref{pre-eq:ws-f} with $f\in L^{\frac{N}{N-\epsilon}}(\Omega)$ for some $0<\epsilon\leq1$. Assume $B_{2R}\subset\Omega$. Then $$\|u^+\|_{L^\infty(B_R)}\leq CR^{-1}\|u^+\|_{L^N(B_{2R})}+CR^{\frac{\epsilon}{N-1}}\|f\|_{ L^{\frac{N}{N-\epsilon}}(B_{2R})}^{\frac{1}{N-1}}$$
where $C=C(N,\epsilon,H)$ is a constant.
\end{lemma}

\begin{proof}
The proof immediately follows by adapting the argument in the proof of \cite[Theorem 2]{Serrin1964} for equation \eqref{pre-eq:ws-f} and replacing $|u|$ with $u^+$ in the proof therein.
\end{proof}

With above lemmas, we prove Proposition \ref{pre-pro:global} in the following. 
\begin{proof}[Proof of Proposition \ref{pre-pro:global}]
Let $\bar{x}\in\mathbb{R}^N$ and $0<r<1$, and let $h\in W^{1,N}(B_r(\bar x ))$ be the weak solution of the problem
\begin{equation*}
\begin{cases}
\Delta_N^H\,h=0&\text{in }B_r(\bar x)\\
h=u&\text{on }\partial B_r(\bar x).
\end{cases}
\end{equation*}
Clearly, $h\leq u$ in $B_r(\bar x)$ by the comparison principle. Then we have
\begin{equation}\label{pre-eq:LN-u-h}
\int_{B_r(\bar x )}(h^+)^N\,dx\leq\int_{B_r(\bar x )}(u^+)^N\,dx\leq N!\int_{\mathbb{R}^N}e^u\,dx.
\end{equation}
Applying Lemma \ref{pre-lem:Serrin} to $h^+$ yields 
\begin{equation}\label{pre-eq:sup-h}
\|h^+\|_{L^\infty(B_{r/2}(\bar x))}\leq C(r).
\end{equation}

Assume now $r$ to be small enough (possibly dependent on $\bar x$) such that 
\begin{equation}\label{pre-eq:mu-r}
\|e^u\|_{L^1(B_r(\bar x ))}^{\frac{1}{N-1}}\leq\frac{(N-1)\mu}{2N}
\end{equation}
where $\mu$ is as in Lemma \ref{pre-lem:priori-exp}. By Lemma \ref{pre-lem:priori-exp} we obtain that
\begin{equation}\label{pre-eq:exp-u-h}
\int_{B_r(\bar x )}e^{\frac{N}{N-1}(u-h)}\,dx\leq C(r).
\end{equation}
Then \eqref{pre-eq:sup-h} yields
\begin{equation}\label{pre-eq:Le-eu}
\int_{B_{r/2}(\bar x)}e^{\frac{N}{N-1}u}\,dx=\int_{B_{r/2}(\bar x)}e^{\frac{N}{N-1}(u-h)}e^{\frac{N}{N-1}h}\,dx\leq C(r).
\end{equation}
Thanks to \eqref{pre-eq:LN-u-h} and \eqref{pre-eq:Le-eu}, by applying Lemma \ref{pre-lem:Serrin} we thus get  
\begin{equation}\label{pre-eq:sup-u+}
\|u^+\|_{L^\infty(B_{r/4}(\bar x))}\leq C(r).
\end{equation} 

Let $R_0>0$ be large enough and to be chosen later. For any $\bar x \in\mathbb{R}^N\setminus\overline{B_{R_0+1}}$ we have that
$$\int_{B_1(\bar x )}e^u\,dx\leq\int_{\mathbb{R}^N\setminus B_{R_0}}e^u\,dx \,,$$
and we deduce that \eqref{pre-eq:mu-r} holds with $r=1$ as long as $R_0$ is sufficiently large. Thus, it follows from \eqref{pre-eq:sup-u+} that 
\begin{equation}\label{pre-eq:sup-u+-ex}
\|u^+\|_{L^\infty\left(\mathbb{R}^N\setminus\overline{B_{R_0+1}}\right)}\leq C(1).
\end{equation}
Since $\overline{B_{R_0+1}}$ is compact, one can find a finite number of points $x_i$ and numbers $r_i$ ($i=1,\cdots,m$) fulfilling \eqref{pre-eq:mu-r} such that 
$$\overline{B_{R_0+1}}\subset\bigcup_{i=1}^m B_{r_i/4}(x_i).$$
 Hence, by \eqref{pre-eq:sup-u+} we have
\begin{equation*}
\|u^+\|_{L^\infty(\overline{B_{R_0+1}})}\leq\max_{1\leq i\leq m}C(r_i)<+\infty.
\end{equation*}
This together with \eqref{pre-eq:sup-u+-ex} implies $u^+\in L^\infty(\mathbb{R}^N)$. 

Furthermore, noting that $e^u\in L^\infty(\mathbb{R}^N)$, we can apply interior regularity results in \cite{DiBen1983,Tolksdorf1984,Serrin1964} to deduce that $u\in C_{\mathrm{loc}}^{1,\alpha}(\mathbb{R}^N)$ for some $\alpha\in(0,1)$. 

Next we show property \eqref{pre-eq:u-upper}. For $R>0$, let $$u_R(x)=u(Rx)+N\log R\quad\quad\forall\,x\in\mathbb{R}^N.$$ Observe that $u_R$ is also a solution of \eqref{in-eq:Lio}. Given $x_0\in\mathbb{S}^{N-1}$, let $h_R\in W^{1,N}(B_\frac12(x_0))$ be the weak solution of
\begin{equation*}
\begin{cases}
\Delta_N^H\,h_R=0&\text{in }B_\frac12(x_0)\\
h_R=u_R&\text{on }\partial B_\frac12(x_0).
\end{cases}
\end{equation*}
As argued above for \eqref{pre-eq:LN-u-h}, we see that 
\begin{equation*}
\int_{B_\frac12(x_0)}(h_R^+)^N\,dx\leq\int_{B_\frac12(x_0)}(u_R^+)^N\,dx\leq N!\int_{\mathbb{R}^N}e^{u_R}\,dx= N!\int_{\mathbb{R}^N}e^{u}\,dx.
\end{equation*}
Thus, Lemma \ref{pre-lem:Serrin} yields that
$$\|h_R^+\|_{L^\infty(B_{1/4}(x_0))}\leq C$$
for some constant $C$ independent of $R$ and of $x_0$.

Since $$\int_{B_\frac12(x_0)}e^{u_R}\,dx\leq\int_{\mathbb{R}^N\setminus B_\frac12}e^{u_R}\,dx\leq\int_{\mathbb{R}^N\setminus B_\frac{R}{2}}e^{u}\,dx,$$
there exists $R_1>1$ sufficiently large such that for any $R>R_1$ we have
\begin{equation*}
\|e^{u_R}\|_{L^1(B_\frac12(x_0))}^{\frac{1}{N-1}}\leq\frac{(N-1)\mu}{2N}.
\end{equation*}
Then we can argue as in \eqref{pre-eq:exp-u-h}--\eqref{pre-eq:sup-u+} to obtain
\begin{equation*}
\|u_R^+\|_{L^\infty(B_{1/8}(x_0))}\leq C
\end{equation*}
for any $R>R_1$ and some constant $C$ independent of $R$ and of $x_0$. This implies that
\begin{equation}\label{pre-eq:sup-S-uR}
\|u_R^+\|_{L^\infty(\mathbb{S}^{N-1})}\leq C,
\end{equation}
since $\mathbb{S}^{N-1}$ can be covered by a finite number of balls $B_{1/8}(x_0)$, $x_0\in\mathbb{S}^{N-1}$.

In view of \eqref{pre-eq:sup-S-uR}, we can deduce that
\begin{equation*}
u(x)+N\log|x|\leq C
\end{equation*}
for all $x\in\mathbb{R}^N\setminus\overline{B_{R_1}}$. Since this estimate also holds for $x\in\overline{B_{R_1}}\setminus\{0\}$, we have shown \eqref{pre-eq:u-upper}, completing the proof.
\end{proof}

Proposition \ref{pre-pro:global} enables us to derive the following inequality, whose argument is essential for proving Theorem \ref{thm:main} in Section \ref{sec:pf}.
 
\begin{lemma}\label{pre-lem:eu-lower}
Let $u$ be a solution of \eqref{in-eq:Lio}. Then
\begin{equation}\label{pre-eq:eu-lower}
\int_{\mathbb{R}^N}e^u\,dx\geq N\left(\frac{N^2}{N-1}\right)^{N-1}\mathcal{L}^N(B_1^{\hat{H}_0}).
\end{equation}
\end{lemma}

\begin{proof}
As in \cite{CL1991,Esposito2018}, we exploit a level set argument.

From \cite[Corollary 1.7]{ACF2021}\footnote{We notice that the anisotropy $H$ considered in \cite{ACF2021} was actually assumed to be a centrally symmetric norm. However, a careful inspection of the proofs reveals that all the results in \cite{ACF2021} remain valid when $H$ is the one we are considering, i.e. a possibly asymmetric norm. In particular, one can repeat all the steps in \cite{ACF2021} in this case. For this reason, we omit the proof.}, we know that the set $$Z_k=\{x\in B_k:\nabla u(x)=0\}$$ is of $\mathcal{L}^N$-measure zero for all $k\in\mathbb{N}$. This implies that the set
$$\{t\in\mathbb{R}:\exists\,x\in\mathbb{R}^N\text{ s.t. }u(x)=t,\nabla u(x)=0\}=\bigcup_{k\in\mathbb{N}}u(Z_k)$$
is of $\mathcal{L}^1$-measure zero. Thus, for almost all $t\in(-\infty,t_0)$ with $t_0:=\sup_{\mathbb{R}^N}u<+\infty$, the superlevel set 
\begin{equation}\label{pre-eq:superlevel-u}
	\Omega_t=\{x\in\mathbb{R}^N:u(x)>t\}
\end{equation}
is a smooth set and the unit outer normal to $\partial\Omega_t$ is $-\frac{\nabla u}{|\nabla u|}$. Also, it follows from \eqref{pre-eq:u-upper} that $\Omega_t$ is bounded for each $t\in(-\infty,t_0)$.

By the divergence theorem (see for instance \cite[Lemma 4.3]{CL2022}), we have that
\begin{align}
\int_{\Omega_t}e^u\,dx&=\int_{\partial\Omega_t}H^{N-1}(\nabla u)\left<\nabla H(\nabla u),\frac{\nabla u}{|\nabla u|}\right>d\mathcal{H}^{N-1}\notag\\
&=\int_{\partial\Omega_t}\frac{H^{N}(\nabla u)}{|\nabla u|}\,d\mathcal{H}^{N-1}\label{pre-eq:level-eu}
\end{align}
for a.e. $t<t_0$. Here we have used that $\left<\nabla H(\xi),\xi\right>= H(\xi)$ for any $\xi\neq 0$, which follows from the homogeneity of $H$. For $\epsilon\in(0,1)$, set $$\Omega_{t,\epsilon}=\{x\in\mathbb{R}^N:u(x)>t,|\nabla u(x)|>\epsilon\}.$$
By the coarea formula, we can calculate
\begin{align*}
\mathcal{L}^N(\Omega_{t,\epsilon})=\int_t^{t_0}\int_{\partial\Omega_s\cap\{|\nabla u|>\epsilon\}}\frac{1}{|\nabla u|}\,d\mathcal{H}^{N-1}\,ds.
\end{align*}
Using the monotone convergence and letting $\epsilon\to 0$ yields
\begin{equation*}
\mathcal{L}^N(\Omega_{t})=\int_t^{t_0}\int_{\partial\Omega_s}\frac{1}{|\nabla u|}\,d\mathcal{H}^{N-1}\,ds
\end{equation*}
for a.e. $t<t_0$. Hence,
\begin{equation}\label{pre-eq:d-level}
-\frac{d}{dt}\mathcal{L}^N(\Omega_{t})=\int_{\partial\Omega_t}\frac{1}{|\nabla u|}\,d\mathcal{H}^{N-1}\quad\text{for a.e. } t<t_0.
\end{equation}
Likewise, we have
\begin{equation*}
\int_{\Omega_t}e^u\,dx=\int_{t}^{t_0}\int_{\partial\Omega_s}\frac{e^s}{|\nabla u|}\,d\mathcal{H}^{N-1}\,ds
\end{equation*}
and
\begin{equation}\label{pre-eq:level-d-eu}
-\frac{d}{dt}\left(\int_{\Omega_t}e^u\,dx\right)=e^t\int_{\partial\Omega_t}\frac{1}{|\nabla u|}\,d\mathcal{H}^{N-1}
\end{equation}
for a.e. $t<t_0$.

By virtue of \eqref{pre-eq:level-eu}--\eqref{pre-eq:level-d-eu}, we compute
\begin{align}
&-\frac{d}{dt}\left(\int_{\Omega_t}e^u\,dx\right)^{\frac{N}{N-1}}\notag\\
&=\frac{N}{N-1}\left(\int_{\partial\Omega_t}\frac{H^{N}(\nabla u)}{|\nabla u|}\,d\mathcal{H}^{N-1}\right)^{\frac{1}{N-1}}e^t\int_{\partial\Omega_t}\frac{1}{|\nabla u|}\,d\mathcal{H}^{N-1}\notag\\
&\geq\frac{N}{N-1}e^t\left(\int_{\partial\Omega_t}\frac{H(\nabla u)}{|\nabla u|}\,d\mathcal{H}^{N-1}\right)^{\frac{N}{N-1}}\label{pre-eq:Hol-ineq}\\
&=\frac{N}{N-1}e^t\left(\int_{\partial\Omega_t}H(-\nu)\,d\mathcal{H}^{N-1}\right)^{\frac{N}{N-1}}\notag\\
&\geq\frac{N}{N-1}e^t\left[N\left[\mathcal{L}^N(B_1^{\hat{H}_0})\right]^{\frac{1}{N}}\left[\mathcal{L}^N(\Omega_{t})\right]^{\frac{N-1}{N}}\right]^{\frac{N}{N-1}}\label{pre-eq:iso-ineq}
\end{align}
for a.e. $t<t_0$, where we have used the H\"older inequality and the anisotropic isoperimetric inequality (see for instance \cite[Theorem 1.2]{CRS2016}) in \eqref{pre-eq:Hol-ineq} and \eqref{pre-eq:iso-ineq}, respectively. For simplicity, let us set 
\begin{equation}\label{pre-eq:def-kappa}
	\kappa_N:=\frac{N^2}{N-1}(N)^{\frac{1}{N-1}}\left[\mathcal{L}^N(B_1^{\hat{H}_0})\right]^{\frac{1}{N-1}}.
\end{equation}
Then, integrating the above inequality from $-\infty$ to $t_0$ we get
\begin{align*}
\left(\int_{\mathbb{R}^N}e^u\,dx\right)^{\frac{N}{N-1}}&\geq\kappa_N\int_{-\infty}^{t_0}e^t\int_{\mathbb{R}^N}\chi_{\Omega_t}\,dx\,dt\\
&=\kappa_N\int_{\mathbb{R}^N}\int_{-\infty}^{t_0}e^t\chi_{\Omega_t}\,dt\,dx\\
&=\kappa_N\int_{\mathbb{R}^N}\int_{-\infty}^{u}e^t\,dt\,dx\\
&=\kappa_N\int_{\mathbb{R}^N}e^u\,dx.
\end{align*}
This implies \eqref{pre-eq:eu-lower} and finishes the proof.
\end{proof}

\begin{remark}\label{pre-rk:non-weight}
	Unlike the argument in \cite[Theorem 3.1]{Esposito2018} leading to \eqref{pre-eq:eu-lower} for $H(\xi)=|\xi|$, here we do not need to rely on a weighted $L^q$ ($1\le q<N$) estimate for $\nabla u$ at infinity (see \cite[Corollary 2.6]{Esposito2018}) to compute integrals over the level sets $\Omega_t$ (such as \eqref{pre-eq:level-eu}) by approximating them on $\Omega_t\cap B_r$ and letting $r\to+\infty$ instead. This is because we already know that $\Omega_t$ is bounded thanks to estimate \eqref{pre-eq:u-upper}, as explained in \cite[Remark 3.2]{Esposito2018}. Also, unlike \cite{Esposito2018}, in the next section we will use estimate \eqref{pre-eq:u-upper} instead of inequality \eqref{pre-eq:eu-lower} to analyze the precise asymptotic behavior of $u$ and $\nabla u$ at infinity.
\end{remark}

\section{Asymptotic behavior at infinity}\label{sec:asym}

Starting from estimate \eqref{pre-eq:u-upper}, in this section we shall determine precisely the asymptotic behavior at infinity for solutions $u$ of \eqref{in-eq:Lio} and that for $\nabla u$ as described in Subsection \ref{sec:derivation}. Hereafter, for simplicity, we set
\begin{equation}\label{asym-eq:def-gamma0}
	\gamma_0:=\left[\frac{\int_{\mathbb{R}^N}e^u\,dx}{N\mathcal{L}^N(B_1^{\hat{H}_0})}\right]^{\frac{1}{N-1}}.
\end{equation}

\begin{proposition}\label{asym-prop:asym-u-Du}
Let $u$ be a solution of \eqref{in-eq:Lio} and let $\gamma_0$ be given by \eqref{asym-eq:def-gamma0}. Then 
\begin{equation}\label{asym-eq:prop-log-u}
	u(x)+\gamma_0\log\hat{H}_0(x)\in L^\infty(\mathbb{R}^N\setminus B_1)
\end{equation}
and
\begin{equation}\label{asym-eq:prop-asym-Du}
\lim_{|x|\to\infty}|x|\left|\nabla\left(u(x)+\gamma_0\log \hat{H}_0(x)\right)\right|=0.
\end{equation}
\end{proposition}

\begin{proof}
	 We shall first show a lower logarithmic estimate for $u$ near infinity. By Proposition \ref{pre-pro:global}, we know $t_0:=\sup_{\mathbb{R}^N} u<+\infty$ and  $$0\le e^{u(x)}\le\frac{C}{|x|^N}\quad\text{in }\mathbb{R}^N\setminus\{0\},$$ for some constant $C>0$. Let $\tilde{u}=t_0-u$. It follows from \eqref{in-eq:Lio} and the positive homogeneity of $H$ that $\tilde{u}$ is a nonnegative solution of the equation
	\begin{equation}\label{asym-eq:Lio-dual-u}
		\mathrm{div}\left(\hat{H}^{N-1}(\nabla\tilde{u})\nabla \hat{H}(\nabla \tilde{u})\right)=e^u \quad\text{in }\mathbb{R}^N,
	\end{equation}
where $\hat{H}$ is as in \eqref{in-eq:def-hat-H}. Notice that $\tilde{u}(x)\to+\infty$ as $|x|\to+\infty$ by \eqref{pre-eq:u-upper}. Thus, by applying Theorem \ref{sing-thm:log-A-f} to \eqref{asym-eq:Lio-dual-u}, we get that estimate \eqref{sing-eq:log-sing} holds for $\tilde{u}$, which in turn implies
\begin{equation}\label{asym-pf-eq:log-u}
	-d\log|x|\le u(x)\le -\frac{1}{d}\log|x|
\end{equation}
in a neighborhood of infinity, for some constant $d>1$.

Let $\gamma>0$ be such that 
\begin{equation}\label{asym-eq:def-gamma}
	-\gamma=\liminf_{|x|\to+\infty}\frac{u(x)}{\log\hat{H}_0(x)}\,.
\end{equation}
 Clearly, $\gamma$ exists and it is unique, according to \eqref{asym-pf-eq:log-u} and the fact that 
\begin{equation}\label{asym-eq:bound-H0}
	\mu_1|x|\le H_0(x)\le\mu_2|x|\quad\,\forall x\in\mathbb{R}^N,
\end{equation}
for some constants $\mu_1,\mu_2>0$ (which follows from the homogeneity of $H_0$). We next demonstrate that, actually, 
\begin{equation}\label{asym-eq:lim-u}
	\lim_{|x|\to+\infty}\frac{u(x)}{\log\hat{H}_0(x)}=-\gamma
\end{equation} 
and that
\begin{equation}\label{asym-eq:lower-u}
	u(x)\ge C_1-\gamma\log\hat{H}_0(x)\quad\text{in }\mathbb{R}^N\setminus\overline{B_1}
\end{equation}
for some constant $C_1$.

 Indeed, for $R>2$ and $|x|>\frac{1}{R}$, let $u_R(x)=\frac{u(Rx)-\omega}{\log R}$ and $f_R(x)=e^{u(Rx)}$. Here $\omega:=\inf_{\partial B_2^{\hat{H}_0}} u$. Via \eqref{asym-pf-eq:log-u}, we see that $$|u_R(x)|\le \frac{d\log|Rx|+C_2}{\log R} $$ for some constant $C_2$, which implies $u_R$ is bounded in $L_{\mathrm{loc}}^\infty(\mathbb{R}^N\setminus\{0\})$, uniformly in $R$. Also, we have
 \begin{equation}\label{asym-eq:uR}
 	-\Delta_N^H \,u_R=\frac{R^N}{(\log R)^{N-1}}f_R\quad\text{in }\mathbb{R}^N\setminus\overline{B_\frac1R}.
 \end{equation}
 Notice from \eqref{pre-eq:u-upper} that $|f_R(x)|\le \frac{C_3}{|Rx|^N}$ for any $x\neq 0$ and for some constant $C_3$, which implies that the right-hand side of \eqref{asym-eq:uR} is also bounded in $L_{\mathrm{loc}}^\infty(\mathbb{R}^N\setminus\{0\})$, uniformly in $R$. Hence, it follows from the regularity results in \cite{DiBen1983,Tolksdorf1984} that $u_R$ is uniformly bounded in $C^{1,\alpha}_{\mathrm{loc}}(\mathbb{R}^N\setminus\{0\})$. Then Ascoli--Arzel\`a theorem and a diagonal process yield that, along a subsequence $R_j\to+\infty$, $u_{R_j}\to u_\infty$ in $C_{\mathrm{loc}}^1(\mathbb{R}^N\setminus\{0\})$ where $u_\infty$ satisfies
 \begin{equation*}
 	\Delta_N^H\, u_\infty=0 \quad \text{in } \mathbb{R}^N\setminus\{0\}.
 \end{equation*}
 Moreover, we observe that $|u_\infty(x)|\le d$ for any $x\neq 0$. Thus, it is forced by Proposition \ref{sing-prop:rigid} (or Lemma \ref{sing-lem:rigid}) that $u_\infty$ is a constant function. To determine this constant, let us define
$$\Gamma(R)=\inf_{2\le \hat{H}_0(x)\le R}\frac{u(x)-\omega}{\log \hat{H}_0(x)}.$$
 Clearly, $\Gamma(R)$ is non-increasing and $\Gamma(R)<0$ for all sufficiently large $R$ (let us say $R>\mathscr{R}$). Since for each $R>\mathscr{R}$, $$-\Delta_N^H\, (u-\omega)\ge-\Delta_N^H\left(\Gamma(R)\log\hat{H}_0\right)=0\quad\text{in } B_R^{\hat{H}_0}\setminus\overline{B_2^{\hat{H}_0}}\,,$$ the comparison principle yields that
 $$\Gamma(R)=\inf_{\partial B_R^{\hat{H}_0}}\frac{u(x)-\omega}{\log R}.$$ Therefore, by the definition of $-\gamma$, we infer that $\lim_{R\to+\infty}\Gamma(R)=-\gamma$, which verifies \eqref{asym-eq:lower-u}. Moreover, for each $R>\mathscr{R}$, let $x_R\in\partial B_1^{\hat{H}_0}$ be such that $\Gamma(R)$ attains its minimum at the point $Rx_R\in\partial B_R^{\hat{H}_0}$. Then $u_R(x_R)=\Gamma(R)$. Assume that $x_{R_j}\to x_\infty$ on $\partial B_1^{H_0}$ up to a subsequence. We derive $$u_\infty(x_\infty)=\lim_{R_j\to+\infty}u_{R_j}(x_{R_j})=\lim_{R_j\to+\infty}\Gamma(R_j)=-\gamma.$$
 Consequently, we conclude from the uniqueness of $-\gamma$ that $$u_{R}\to -\gamma\,\text{ in }\, C_{\mathrm{loc}}^1(\mathbb{R}^N\setminus\{0\}),\quad\text{as } R\to+\infty,$$ (since this holds for any sequence $R\to+\infty$ up to extracting a subsequence). In particular, it implies \eqref{asym-eq:lim-u}.

 We proceed to show that an upper counterpart of \eqref{asym-eq:lower-u} holds as well, i.e., there is a constant $C_4$ such that \begin{equation}\label{asym-eq:upper-u}
	u(x)\le C_4-\gamma\log \hat{H}_0(x)\quad\text{in }\mathbb{R}^N\setminus\overline{B_1}.
\end{equation}
To this aim, observe first that 
$$
\gamma>N
$$ 
as it follows by \eqref{asym-eq:lower-u} together with the fact $\int_{\mathbb{R}^N}e^u\,dx<+\infty$. Fix $0<\epsilon_0\ll 1$ such that $-\gamma_1:=-\gamma+\epsilon_0<-N$. Then \eqref{asym-eq:lim-u} yields that
\begin{equation}\label{asym-eq:upper2-u}
	u(x)\le C_5-\gamma_1\log\hat{H}_0(x)\quad\text{in }\mathbb{R}^N\setminus\overline{B_1},
\end{equation}
for some constant $C_5$. This improves the known upper estimate \eqref{pre-eq:u-upper}. Now, we further sharpen it to \eqref{asym-eq:upper-u} by constructing appropriate supersolutions to \eqref{in-eq:Lio} in an exterior domain and then using the comparison principle, as explained in the following.

In order to get a supersolution of \eqref{in-eq:Lio}, we introduce $\phi=\phi(t)$ which is a smooth function for $t>0$ such that $\phi':=\frac{d\phi}{dt}<0$. Setting $v(x)=\phi(\hat{H}_0(x))$ for $x\in\mathbb{R}^N\setminus\{0\}$, it is easily verified that $$-\Delta_N^H \,v=t^{1-N}\left(t^{N-1}(-\phi')^{N-1}\right)'.$$ Notice from \eqref{asym-eq:upper2-u} that $e^{u(x)}\le C_6\hat{H}_0^{-\gamma_1}(x)$ for $|x|>1$ and $C_6>0$ being a constant. Thus, given $t_0>0$, let us consider the ODE
 \begin{equation}\label{asym-eq:ODE}
 		t^{1-N}\left(t^{N-1}(-\phi')^{N-1}\right)'=t^{-\gamma_1}\quad\text{for }t>t_0.\\
 \end{equation}
By integration, it can be deduced that the function
\begin{equation*}
	\phi_{a,b}(t)=-(\gamma_1-N)^{\frac{1}{1-N}}\int_{t_0}^t\frac{(a-\tau^{N-\gamma_1})^{\frac{1}{N-1}}}{\tau}\,d\tau+b
\end{equation*}
solves \eqref{asym-eq:ODE} for any constants $a\ge (t_0)^{N-\gamma_1}$ and $b\in\mathbb{R}$.

Now, let $v_{a,b}(x)=C_6^{\frac{1}{N-1}}\phi_{a,b}(\hat{H}_0(x))$. When $t_0>\mu_2$ where $\mu_2$ is as in \eqref{asym-eq:bound-H0}, we have
\begin{equation}\label{asym-eq:v-u}
	-\Delta_N^H\,v_{a,b}=C_6\hat{H}_0^{-\gamma_1}\ge e^u=-\Delta_N^H\,u\quad\text{in }\mathbb{R}^N\setminus\overline{B_{t_0}^{\hat{H}_0}}.
\end{equation}
Also, we find that, for $x$ with $\hat{H}_0(x)\ge t_0$,
\begin{align}
	v_{a,b}(x)&\ge-\left(\frac{aC_6}{\gamma_1-N}\right)^{\frac{1}{N-1}}\int_{t_0}^{\hat{H}_0(x)} \frac{1}{\tau}\,d\tau+C_6^{\frac{1}{N-1}}b\notag\\
	&=-\left(\frac{aC_6}{\gamma_1-N}\right)^{\frac{1}{N-1}}\log \hat{H}_0(x)\,+ C_7+C_6^{\frac{1}{N-1}}b \label{asym-eq:C7}
\end{align}
where $C_7=C_7(a,t_0)$ is a constant. In particular, we have 
\begin{equation}\label{asym-eq:v-t0}
	v_{a,b}(x)= C_6^{\frac{1}{N-1}}b\quad\text{when }\hat{H}_0(x)=t_0.
\end{equation}
 With these, let us show by choosing suitable $a$ and $b$ that $v_{a,b}(x)\ge u(x)$ for $x$ satisfying $\hat{H}_0(x)\ge t_0$.

Indeed, for each $\epsilon\in(0,\epsilon_0)$, \eqref{asym-eq:lim-u} implies that there is $R(\epsilon)$ such that
\begin{equation}\label{asym-eq:R-e-u}
	u(x)\le (-\gamma+\epsilon)\log \hat{H}_0(x)\quad\text{for }|x|>R(\epsilon).
\end{equation}
Take $$a_\epsilon=\frac{(\gamma-\epsilon)^{N-1}(\gamma_1-N)}{C_6}$$
and notice that $$\frac{\gamma_1^{N-1}(\gamma_1-N)}{C_6}\le a_\epsilon\le\frac{\gamma^{N-1}(\gamma_1-N)}{C_6}.$$
Thus, by letting $t_0$ be sufficiently large (but still fixed), we have $a_\epsilon\ge 2(t_0)^{N-\gamma_1}$ for any $\epsilon\in(0,\epsilon_0)$. Also, it follows that the constant $C_7(a_\epsilon,t_0)$ as in \eqref{asym-eq:C7} is uniformly bounded. Now, fix $b$ such that
$$C_6^{\frac{1}{N-1}}b\ge\max\{-C_7(a_\epsilon,t_0)\,,\,\max_{\partial B_{t_0}^{\hat{H}_0}}u\}$$
holds for any $\epsilon\in(0,\epsilon_0)$. In view of \eqref{asym-eq:C7}, \eqref{asym-eq:v-t0} and \eqref{asym-eq:R-e-u}, we obtain, for each $\epsilon\in(0,\epsilon_0)$,
\begin{equation*}
	v_{a_\epsilon,b}(x)\ge u(x)\quad\text{when }\hat{H}_0(x)= t_0\text{ and when }|x|>R(\epsilon).
\end{equation*}
 Recall that \eqref{asym-eq:v-u} holds. By using the comparison principle, we thus infer that
 $$v_{a_\epsilon,b}(x)\ge u(x)\quad\text{on }\mathbb{R}^N\setminus B_{t_0}^{\hat{H}_0}$$
for each $\epsilon\in(0,\epsilon_0)$. 

Note that, for $\epsilon\in(0,\epsilon_0)$,
\begin{align*}
	v_{a_\epsilon,b}(x)-C_6^{\frac{1}{N-1}}b&=-\left(\frac{C_6}{\gamma_1-N}\right)^{\frac{1}{N-1}}\int_{t_0}^{\hat{H}_0(x)} \frac{(a_\epsilon-\tau^{N-\gamma_1})^{\frac{1}{N-1}}}{\tau}\,d\tau\\
	&=-\left(\frac{a_\epsilon C_6}{\gamma_1-N}\right)^{\frac{1}{N-1}}\int_{t_0}^{\hat{H}_0(x)} \frac{1}{\tau}\,d\tau+I\\
	&=(-\gamma+\epsilon)(\log\hat{H}_0(x)-\log t_0)+I
\end{align*}
where $$I=\left(\frac{C_6}{\gamma_1-N}\right)^{\frac{1}{N-1}}\int_{t_0}^{\hat{H}_0(x)} \frac{(a_\epsilon)^{\frac{1}{N-1}}-(a_\epsilon-\tau^{N-\gamma_1})^{\frac{1}{N-1}}}{\tau}\,d\tau.$$
Since for $\tau>t_0$, 
\begin{align*}
	0\le(a_\epsilon)^{\frac{1}{N-1}}-(a_\epsilon-\tau^{N-\gamma_1})^{\frac{1}{N-1}}&\le\frac{\tau^{N-\gamma_1}}{N-1}(a_\epsilon-\tau^{N-\gamma_1})^{\frac{2-N}{N-1}}\\
	&\le\frac{t_0^{\frac{(N-\gamma_1)(2-N)}{N-1}}}{N-1}\tau^{N-\gamma_1}
\end{align*}
by Bernoulli's inequality and the fact that $a_\epsilon\ge 2(t_0)^{N-\gamma_1}$, we deduce that $I\le C_8$ for some constant $C_8$ independent of $\epsilon$. Hence, we conclude that $$u(x)\le v_{a_\epsilon,b}(x)\le(-\gamma+\epsilon)(\log\hat{H}_0(x)-\log t_0)+C_8+C_6^{\frac{1}{N-1}}b$$
for $\hat{H}_0(x)\ge t_0$. By letting $\epsilon\to0$ in the above we obtain \eqref{asym-eq:upper-u}.

At this point, in view of \eqref{asym-eq:lower-u} and \eqref{asym-eq:upper-u}, we have arrived at 
\begin{equation}\label{asym-eq:log-u}
	u(x)=-\gamma\log \hat{H}_0(x)+G(x)\quad\text{for }|x|\ge1,
\end{equation}
where $G(x)\in L^\infty(\mathbb{R}^N\setminus B_1)$. Next, for $m\in \mathbb N$, we introduce the function $$u_m(y):=u(my)+\gamma\log m.$$
Then it follows from \eqref{asym-eq:log-u} that $$u_m(y)=-\gamma\log \hat{H}_0(y)+G_m(y)$$
and \begin{equation}\label{asym-eq:um-eq}
	-\Delta_N^H\,u_m=m^{N-\gamma}\frac{e^{G_m}}{\hat{H}_0^\gamma}
\end{equation}
where $G_m(y):=G(my)$. Observe that both $u_m$ and the right-hand side of \eqref{asym-eq:um-eq} are bounded in $L_{\mathrm{loc}}^\infty(\mathbb{R}^N\setminus\{0\})$, uniformly in $m$. Thus, by the regularity results in  \cite{Serrin1964,DiBen1983,Tolksdorf1984}, we infer that $u_m$ is uniformly bounded in $C^{1,\alpha}_{\mathrm{loc}}(\mathbb{R}^N\setminus\{0\})$. By Ascoli--Arzel\`a theorem and a diagonal process, we therefore have that $u_m\to U_\infty$ in $C_{\mathrm{loc}}^1(\mathbb{R}^N\setminus\{0\})$, along a subsequence of $m\to\infty$. Here $U_\infty$ satisfies $$\Delta_N^H\,U_\infty=0\quad\text{in }\mathbb{R}^N\setminus\{0\}.$$
Also, we can find $$U_\infty(y)=-\gamma\log \hat{H}_0(y)+G_\infty(y)$$
where $G_\infty$ is the limit of $G_m$ in $C_{\mathrm{loc}}^1(\mathbb{R}^N\setminus\{0\})$. The uniform boundedness of $G_m$ shows $G_\infty\in L^\infty(\mathbb{R}^N)$. Hence, Proposition \ref{sing-prop:rigid} forces that $G_\infty$ is a constant function.

Via \eqref{asym-eq:log-u}, we have
\begin{equation*}
	\sup_{|x|=m}|x|\left|\nabla\left(u(x)+\gamma\log \hat{H}_0(x)\right)\right|=\sup_{|y|=1}|\nabla G_m(y)|.
\end{equation*}
Since
\begin{equation*}
	\sup_{|y|=1}|\nabla G_m(y)|\to \sup_{|y|=1}|\nabla G_\infty(y)|=0
\end{equation*}
holds for any sequence $m\to\infty$ up to extracting a subsequence, we deduce that
\begin{equation}\label{asym-eq:asym-Du}
	\lim_{|x|\to\infty}|x|\left|\nabla\left(u(x)+\gamma\log \hat{H}_0(x)\right)\right|=0.
\end{equation}
Consequently, if we can show $\gamma=\gamma_0$, then the validity of \eqref{asym-eq:prop-log-u} and \eqref{asym-eq:prop-asym-Du} follows from \eqref{asym-eq:log-u} and \eqref{asym-eq:asym-Du}.

Indeed, by \eqref{asym-eq:asym-Du} we infer that
\begin{gather*}
	\nabla u=-\gamma\nabla(\log \hat{H}_0(x))+o(\hat{H}_0^{-1}(x))\\
	H(\nabla u)=\gamma\hat{H}_0^{-1}(x)+o(\hat{H}_0^{-1}(x))
\end{gather*}
uniformly for $x\in\partial B_R^{\hat{H}_0}$, as $R\to\infty$. By the divergence theorem, we have
\begin{align}
	\int_{B_R^{\hat{H}_0}}e^u\,dx&=\int_{\partial B_R^{\hat{H}_0}}H^{N-1}(\nabla u)\left<\nabla H(\nabla u),-\nu\right>d\mathcal{H}^{N-1}\notag\\
	&=\int_{\partial B_R^{\hat{H}_0}}H^{N-1}(\nabla u)\left<\nabla H(\nabla u),-\frac{\nabla \hat{H}_0(x)}{|\nabla \hat{H}_0(x)|}\right>d\mathcal{H}^{N-1}\notag\\
	&=\frac{R}{\gamma}\int_{\partial B_R^{\hat{H}_0}}H^{N-1}(\nabla u)\left<\nabla H(\nabla u),\frac{\nabla u+o(R^{-1})}{|\nabla \hat{H}_0(x)|}\right>d\mathcal{H}^{N-1}\notag\\
	&=\frac{R}{\gamma}\int_{\partial B_R^{\hat{H}_0}}H^{N-1}(\nabla u)\frac{H(\nabla u)+o(R^{-1})}{|\nabla \hat{H}_0(x)|}\,d\mathcal{H}^{N-1}\notag\\
	&=\int_{\partial B_R^{\hat{H}_0}}\left(\gamma R^{-1}+o(R^{-1})\right)^{N-1}(1+o(1))H(-\nu)\,d\mathcal{H}^{N-1}.\label{asym-eq:compute}
\end{align} 
Letting $R
\to\infty$ gives $$\int_{\mathbb{R}^N}e^u\,dx=\gamma^{N-1}\int_{\partial B_1^{\hat{H}_0}}H(-\nu)\,d\mathcal{H}^{N-1}=\gamma^{N-1}N\mathcal{L}^{N}(B_1^{\hat{H}_0}),$$
which implies $\gamma=\gamma_0$. This completes the proof.
\end{proof}

\section{Quantization result}\label{sec:quant}
Thanks to Proposition \ref{asym-prop:asym-u-Du}, in this section we obtain a quantization result by exploiting a Pohozaev type identity, which shows that any solution of \eqref{in-eq:Lio} actually attains the equality in \eqref{pre-eq:eu-lower}. 

\begin{proposition}\label{quan-pro}
	Let $u$ be a solution of \eqref{in-eq:Lio}. Then
	\begin{equation*}
		\int_{\mathbb{R}^N}e^u\,dx= N\left(\frac{N^2}{N-1}\right)^{N-1}\mathcal{L}^N(B_1^{\hat{H}_0}).
	\end{equation*}
\end{proposition}

\begin{proof}
	Let $R>0$ and let $y \in \mathbb{R}^N$ be fixed. The divergence theorem yields
	\begin{equation*}
	 \int_{\partial B_{R}^{\hat{H}_0}}e^u\left<x-y,\nu\right>d\mathcal{H}^{N-1}(x)=\int_{B_{R}^{\hat{H}_0}}e^u\left<x-y,\nabla u\right>dx+N\int_{B_{R}^{\hat{H}_0}}e^u\,dx,
	\end{equation*}
where $\nu$ is the unit outer normal to $\partial B_{R}^{\hat{H}_0}$. Then by applying Lemma \ref{quant-lem:Pohozaev} below to equation \eqref{in-eq:Lio}, we deduce that the following identity
	\begin{align}
		&N\int_{B_{R}^{\hat{H}_0}}e^u\,dx-\int_{\partial B_{R}^{\hat{H}_0}}e^u\left<x-y,\nu\right>d\mathcal{H}^{N-1}(x)\notag\\
		=&\int_{\partial B_{R}^{\hat{H}_0}}\left[H^{N-1}(\nabla u)\left<\nabla H(\nabla u),\nu\right>\left<x-y,\nabla u\right>-\frac{H^N(\nabla u)}{N}\left<x-y,\nu\right>\right]d\mathcal{H}^{N-1}(x)\label{quant-eq:Pohozaev-u}
	\end{align}
	 holds.
	 
	Via \eqref{asym-eq:prop-asym-Du}, we see that
	\begin{equation*}
		H(\nabla u)=\gamma_0 \hat{H}_0^{-1}(x)+o(\hat{H}_0^{-1}(x))\quad\text{and}\quad
		\left<x,\nabla u\right>=-\gamma_0+o(1),
	\end{equation*}
	uniformly for $x\in\partial B_{R}^{\hat{H}_0}$, as $R\to\infty$. Hence, as done in \eqref{asym-eq:compute}, we infer that the right-hand side of \eqref{quant-eq:Pohozaev-u} with $y=0$ goes to $$\gamma_0^N(N-1)\mathcal{L}^N(B_1^{\hat{H}_0})$$ as $R\to\infty$. On the other hand, by \eqref{asym-eq:prop-log-u} one has $$e^u\leq C\hat{H}_0^{-\gamma_0}(x)$$ for 
	some constant $C$, as $|x|\to\infty$. Since $\gamma_0\geq\frac{N^2}{N-1}>N$ by \eqref{pre-eq:eu-lower}, we obtain that $$\int_{\partial B_{R}^{\hat{H}_0}}e^u\left<x,\nu\right>d\mathcal{H}^{N-1}\leq CR^{N-\gamma_0}\to 0,\quad\text{as }R\to\infty.$$
 
 As a consequence, by letting $R\to\infty$ in \eqref{quant-eq:Pohozaev-u} with $y=0$ we conclude that 
 $$N\int_{\mathbb{R}^N}e^u\,dx=\gamma_0^N(N-1)\mathcal{L}^N(B_1^{\hat{H}_0}).$$
The assertion is thus proved by recalling \eqref{asym-eq:def-gamma0}.
\end{proof}

In the above proof, we have used the following identity of Pohozaev type. Note that in the Euclidean case, such an identity has been established in \cite{DFSV2009,IT2012}.
\begin{lemma}\label{quant-lem:Pohozaev}
	Let $\Omega\subset\mathbb{R}^N$ be a bounded open set with Lipschitz boundary and let $f\in L^1(\Omega)\cap L_{\mathrm{loc}}^\infty(\Omega)$. Assume that $p>1$ and $u\in C^1(\overline{\Omega})$ satisfies
	\begin{equation*}
		-\Delta_p^Hu=f\quad\text{in }\Omega.
	\end{equation*}
	Then \begin{align*}
		&\frac{p-N}{p}\int_{\Omega}H^p(\nabla u)\,dx-\int_{\Omega}f(x)\left<x-y,\nabla u\right>dx\\
		=&\int_{\partial\Omega}\left[H^{p-1}(\nabla u)\left<\nabla H(\nabla u),\nu\right>\left<x-y,\nabla u\right>-\frac{H^p(\nabla u)}{p}\left<x-y,\nu\right>\right]d\mathcal{H}^{N-1}(x)
	\end{align*}
for any $y\in\mathbb{R}^N$, where $\nu$ is the unit outer normal to $\partial\Omega$.
\end{lemma}

\begin{proof}
Given $y\in\mathbb{R}^N$, we first show that
\begin{align}
	&\mathrm{div}\left(\left<x-y,\nabla u\right>H^{p-1}(\nabla u)\nabla H(\nabla u)-\frac{x-y}{p}H^p(
	\nabla u)\right)\notag\\
	=&\frac{p-N}{p}H^p(\nabla u)-f(x)\left<x-y,\nabla u\right>\label{quant-eq:loc-Poh}
\end{align}
holds in the sense of distributions in $\Omega$, that is 
\begin{align}
	&\int_{\Omega}\left[H^{p-1}(\nabla u)\left<\nabla H(\nabla u),\nabla\varphi\right>\left<x-y,\nabla u\right>-\frac{H^p(
		\nabla u)}{p}\left<x-y,\nabla\varphi\right>\right]dx\notag\\
	=&\int_{\Omega}f(x)\varphi(x)\left<x-y,\nabla u\right>dx-\frac{p-N}{p}\int_{\Omega}H^p(\nabla u)\varphi\,dx\label{quant-eq:dis-Poh}
\end{align}
for every $\varphi\in C_c^\infty(\Omega)$.

From \cite[Theorem 1.1]{ACF2021} one has $$H^{p-1}(\nabla u)\nabla H(\nabla u)\in W_{\mathrm{loc}}^{1,2}(\Omega).$$
Since $H_0(H^{p-1}(\nabla u)\nabla H(\nabla u))=H^{p-1}(\nabla u)$ by \cite[Formula (2.2)]{CL2022} and $H_0$ is Lipschitz continuous on $\mathbb{R}^N$, the chain rule entails $$H^{p-1}(\nabla u)\in W_{\mathrm{loc}}^{1,2}(\Omega).$$
For $\delta>0$ and $x\in\Omega$, we let 
$$
\psi_{\delta}(x):=\min\{\frac{[H^{p-1}(\nabla u(x))-\delta^{p-1}]^+}{\delta^{p-1}},1\}.
$$ 
It then follows that $\psi_\delta(x)\in W_{\mathrm{loc}}^{1,2}(\Omega)\cap C(\overline{\Omega})$. By setting 
\begin{equation} \label{Udelta}
U^\delta:=\{x\in\Omega:H(\nabla u(x))\geq2^{\frac{1}{p-1}}\delta\}
\end{equation}
and 
\begin{equation*} 
Z:=\{x\in\Omega:\nabla u(x)=0\} \,,
\end{equation*}
we have
\begin{equation*}
	\psi_{\delta}(x)=
	\begin{cases}
			1 &  \text{if }x\in U^\delta \\
		0 & \text{if }x\in Z,
	\end{cases}
\end{equation*}
and we find that $\psi_\delta(x)\to\psi_0(x)$ as $\delta\to0$ for every $x\in\Omega$, where 
\begin{equation*}
	\psi_0(x):=
	\begin{cases}
			1 &  \text{if }x\in \Omega\setminus Z\\
		0 & \text{if }x\in Z \,.
	\end{cases}
\end{equation*}

Given $\varphi\in C_c^\infty(\Omega)$, we can decompose $\varphi$ as the sum 
\begin{equation}\label{quant-eq:phi-decompose}
	\varphi=\psi_\delta\varphi+(1-\psi_\delta)\varphi
\end{equation}
with $\psi_\delta\varphi \in W_{0}^{1,2}(\Omega\setminus Z)$ and $(1-\psi_\delta)\varphi\in W_0^{1,2}(\Omega\setminus U^{2\delta})$; here $U^{2\delta}$ is as defined in \eqref{Udelta}. By standard regularity theory for elliptic equations, one has  $u\in W_{\mathrm{loc}}^{2,2}(\Omega\setminus Z)$. Then a direct computation using the product rule of Sobolev functions shows that \eqref{quant-eq:loc-Poh} holds pointwise in the domain $\Omega\setminus Z$. Hence, we see that \eqref{quant-eq:dis-Poh} is valid for $\psi_\delta\varphi$ in place of $\varphi$. 

On the other hand, we set 
$$
\mathcal{V}(x):=\left<x-y,\nabla u\right>H^{p-1}(\nabla u)\nabla H(\nabla u)-\frac{x-y}{p}H^p(
\nabla u) \,,
$$ 
and for $\delta\in(0,1)$ we have
\begin{align*}
	&\int_{\Omega}\left<\mathcal{V},\nabla((1-\psi_\delta)\varphi)\right>dx\notag\\
	\lesssim &\int_{\Omega\setminus U^\delta}H^p(\nabla u)|x-y|\left(|\nabla\varphi|+|\varphi||\nabla\psi_\delta|\right)dx\notag\\
	\lesssim&\,\delta\int_{\Omega\setminus U^\delta}\delta^{-p}H^p(\nabla u)|x-y|\left(\delta^{p-1}|\nabla\varphi|+\delta^{p-1}|\varphi||\nabla\psi_\delta|\right)dx\notag\\
	\lesssim&\,\delta\int_{\Omega\setminus U^\delta}|x-y|\left(|\nabla\varphi|+\delta^{p-1}|\varphi||\nabla\psi_\delta|\right)dx\\
	\lesssim&\,\delta\int_{\Omega\setminus U^\delta}|x-y|\left(|\nabla\varphi|+|\varphi||\nabla(H^{p-1}(\nabla u))|\right)dx,
\end{align*}
which implies
\begin{equation*}
	\int_{\Omega}\left<\mathcal{V},\nabla((1-\psi_\delta)\varphi)\right>dx\to 0\quad\text{as }\delta\to 0.
\end{equation*}

 Thus, by \eqref{quant-eq:phi-decompose} we conclude that 
\begin{align*}
	\int_{\Omega}\left<\mathcal{V},\nabla\varphi\right>dx=&\lim_{\delta\to 0}\int_{\Omega}\left<\mathcal{V},\nabla(\psi_{\delta}\varphi)+\nabla((1-\psi_\delta)\varphi)\right>dx\\
	=&\lim_{\delta\to 0}\int_{\Omega}\left<\mathcal{V},\nabla(\psi_{\delta}\varphi)\right>dx\\
	=&\lim_{\delta\to 0}\int_{\Omega}\left[f(x)\left<x-y,\nabla u\right>-\frac{p-N}{p}H^p(\nabla u)\right]\psi_\delta\varphi\,dx\\
	=&\int_{\Omega}\left[f(x)\left<x-y,\nabla u\right>-\frac{p-N}{p}H^p(\nabla u)\right]\varphi\,dx.
\end{align*}
This proves \eqref{quant-eq:dis-Poh}. 

Now, since \eqref{quant-eq:loc-Poh} holds in the sense of distributions, the desired identity follows by applying \cite[Lemma 4.3]{CL2022}.
\end{proof}

\section{Proof of Theorem \ref{thm:main}}\label{sec:pf}
The proof is mainly based on the characterization of minimizers of the anisotropic isoperimetric inequality and H\"older inequality which, together with Proposition \ref{quan-pro},  enable us to deduce that all the superlevel sets of the solution $u$ are concentric Wulff shapes associated with $\hat{H}_0$ and that $H(\nabla u)$ is a positive constant on each level set of $u$. These rigidity properties will force $u$ to have a specific form \eqref{in-eq:u} via an ODE argument as in \cite{Esposito2018}.
 
\begin{proof}[Proof of Theorem \ref{thm:main}]
	Since the equality sign in \eqref{pre-eq:eu-lower} holds by Proposition \ref{quan-pro}, we deduce from the proof of Lemma \ref{pre-lem:eu-lower} that we are actually having equalities in \eqref{pre-eq:Hol-ineq} and \eqref{pre-eq:iso-ineq}. Thus, by the characterization of equality cases of the anisotropic isoperimetric inequality and H\"older inequality, we conclude that for a.e. $t<t_0$
	\begin{itemize}
		\item[(i)] $\Omega_t=B_{R(t)}^{\hat{H}_0}(x(t))$ for some $R(t)>0$ and $x(t)\in\mathbb{R}^N$,
		\item[(ii)] $\frac{H^N(\nabla u)}{|\nabla u|}$ is a multiple of $\frac{1}{|\nabla u|}$ on $\partial\Omega_t$.
	\end{itemize}
Here, $\Omega_t$ is given by \eqref{pre-eq:superlevel-u}. 

We first claim that (i) holds for any $t<t_0$. Indeed, let us denote by $\Theta$ the set of $t\in(-\infty,t_0)$ for which (i) holds. Given $t<t_0$, we can find a sequence $t_n\in\Theta$ such that $t_n\searrow t$ and $$\Omega_t=\bigcup_{t_n}\Omega_{t_n}.$$
Since $\Omega_{t_n}=B_{R(t_n)}^{\hat{H}_0}(x(t_n))$, we infer that there are also  $R(t)>0$ and $x(t)\in\mathbb{R}^N$ such that $\Omega_t=B_{R(t)}^{\hat{H}_0}(x(t))$.\footnote{Notice that $\Omega_t$ is a minimizer of the anisotropic isoperimetric inequality, as it follows by observing that $P_{\hat{H}}(\Omega_t;\mathbb{R}^N)\leq\liminf P_{\hat{H}}(\Omega_{t_n};\mathbb{R}^N)$, where $P_{\hat{H}}(\Omega; \mathbb{R}^N)$ denotes the anisotropic perimeter of a measurable set $\Omega\subset\mathbb{R}^N$ with respect to $\hat{H}$ (see for instance \cite[Formula (2.8)]{CL2022} with $D=\mathbb{R}^N$ there).} This proves the claim.

By (ii) and the continuity of $\nabla u$, we infer that $H(\nabla u)$ is a positive constant on $\partial\Omega_t$ for any $t<t_0$. 

Next, we shall derive the explicit representation of the radius $R(t)$ appearing above. 

Set 
$$
M(t)=\int_{\Omega_t}e^u\,dx \,.
$$
By virtue of \eqref{pre-eq:level-eu} and \eqref{pre-eq:iso-ineq} (where now the equality sign holds), we have that 
\begin{align}
	M(t)&=H^{N-1}(\nabla u)\int_{\partial B_{R(t)}^{\hat{H}_0}(x(t))}H(-\nu)\,d\mathcal{H}^{N-1}\notag\\
	&=NH^{N-1}(\nabla u)\mathcal{L}^N(B_1^{\hat{H}_0})R^{N-1}(t)\label{pf-eq:M-R}
\end{align}
 for any $t<t_0$ and that
\begin{align}
	\frac{d}{dt}M^{\frac{N}{N-1}}(t)&=\frac{N}{N-1}M^{\frac{1}{N-1}}(t)M'(t)\notag\\
	&=-\frac{N^2e^t}{N-1}(N)^{\frac{1}{N-1}}\left[\mathcal{L}^N(B_1^{\hat{H}_0})\right]^{\frac{N}{N-1}}R^N(t)\label{pf-eq:D-M}
\end{align}
for a.e. $t<t_0$. Moreover, by using the identity \eqref{quant-eq:Pohozaev-u} with $B_R^{\hat{H}_0}$ replaced by $\Omega_t=B_{R(t)}^{\hat{H}_0}(x(t))$ and with $y=x(t)$, we deduce that 
\begin{equation}\label{pf-eq:Poho-M}
	M(t)=e^t\mathcal{L}^N(B_1^{\hat{H}_0})R^N(t)+\frac{N-1}{N}H^N(\nabla u)\mathcal{L}^N(B_1^{\hat{H}_0})R^N(t).
\end{equation}
Here we have used the homogeneity of $\hat{H}_0$ and the fact that $$u(z)=t\quad\text{and}\quad\nu(z)=-\frac{\nabla u(z)}{|\nabla u(z)|}=\frac{\nabla\hat{H}_0(z-x(t))}{|\nabla\hat{H}_0(z-x(t))|}\quad\text{for } z\in\partial\Omega_t,$$
where $\nu$ is the unit outer normal to $\partial\Omega_t$.

Combining \eqref{pf-eq:M-R} and \eqref{pf-eq:Poho-M} yields 
\begin{equation}\label{pf-eq:eR-M}
	e^tR^N(t)=\frac{M(t)}{\mathcal{L}^N(B_1^{\hat{H}_0})}-\frac{N-1}{N}\left[\frac{M(t)}{N\mathcal{L}^N(B_1^{\hat{H}_0})}\right]^{\frac{N}{N-1}}.
\end{equation}
 Inserting this into \eqref{pf-eq:D-M}, one gets
\begin{equation}\label{pf-eq:d-M-t}
	\frac{N}{N-1}M'(t)=M(t)-\frac{N^2}{N-1}(N)^{\frac{1}{N-1}}\left[\mathcal{L}^N(B_1^{\hat{H}_0})\right]^{\frac{1}{N-1}}M^{\frac{N-2}{N-1}}
\end{equation}
for a.e. $t<t_0$. Note that $M(t_0)=0$. Then we integrate \eqref{pf-eq:d-M-t} and find
\begin{equation}\label{pf-eq:def-M}
	M(t)=\left[\kappa_N\left(1-e^{\frac{t-t_0}{N}}\right)\right]^{N-1}
\end{equation}
for any $t\leq t_0$, where $\kappa_N$ is given by \eqref{pre-eq:def-kappa}. Hence, by combining \eqref{pf-eq:eR-M} and \eqref{pf-eq:def-M} we derive that
\begin{equation}\label{pf-eq:R}
	R^N(t)=N\left(\frac{N^2}{N-1}\right)^{N-1}\left(1-e^{\frac{t-t_0}{N}}\right)^{N-1}e^{-\frac{(N-1)t+t_0}{N}}
\end{equation}
for any $t\leq t_0$.

Now we show that the center $x(t)$ of $B_{R(t)}^{\hat{H}_0}(x(t))$ does not depend on $t$, i.e. there is $x_0\in\mathbb{R}^N$ such that $x(t)\equiv x_0$ for $t<t_0$. Since $$B_{R(t)}^{\hat{H}_0}(x(t))=\Omega_t\subset \Omega_s=B_{R(s)}^{\hat{H}_0}(x(s))$$ for $t>s$, we deduce that
$$\hat{H}_0(x(t)-x(s))\leq C |R(t)-R(s)|\quad\text{for any }t,s<t_0$$
where $C>0$ is a constant depending only on $H_0$.
Thus, $x(t)$ is locally Lipschitz continuous on $(-\infty,t_0)$. 

Assume that there exists $t_1<t_0$ such that $x_1:=\frac{d}{dt}x(t_1)$ exists and does not vanish. For $t<t_0$, let 
$$P(t)=x(t)+R(t)\tau^+\quad\text{and}\quad Q(t)=x(t)+R(t)\tau^{-}$$
where $\tau^+:=\frac{x_1}{\hat{H}_0(x_1)}$ and $\tau^{-}:=\frac{-x_1}{H_0(x_1)}$.
Since $P(t),Q(t)\in\partial\Omega_t$, $u(P(t))=u(Q(t))=t$ and 
\begin{gather}
	\nu(P(t))=-\frac{\nabla u(P(t))}{|\nabla u(P(t))|}=\frac{\nabla\hat{H}_0(P(t)-x(t))}{|\nabla\hat{H}_0(P(t)-x(t))|}=\frac{\nabla\hat{H}_0(\tau^+)}{|\nabla\hat{H}_0(\tau^+)|}\label{pf-eq:nu-P}\\
	\nu(Q(t))=-\frac{\nabla u(Q(t))}{|\nabla u(Q(t))|}=\frac{\nabla\hat{H}_0(Q(t)-x(t))}{|\nabla\hat{H}_0(Q(t)-x(t))|}=\frac{\nabla\hat{H}_0(\tau^{-})}{|\nabla\hat{H}_0(\tau^{-})|}.\label{pf-eq:nu-Q}
\end{gather}
From the homogeneity of $H$ and \cite[Lemma 2.2]{CL2022}, one has
 \begin{gather}
 	H(-\nu(P(t)))=\frac{H(\nabla u(P(t)))}{|\nabla u(P(t))|}=\frac{1}{|\nabla\hat{H}_0(\tau^{+})|},\label{pf-eq:hH-nu-P}\\
 	H(-\nu(Q(t)))=\frac{H(\nabla u(Q(t)))}{|\nabla u(Q(t))|}=\frac{1}{|\nabla\hat{H}_0(\tau^{-})|}.\label{pf-eq:hH-nu-Q}
 \end{gather}
 Then, since $\langle\nabla\hat{H}_0(\xi),\xi\rangle=\hat{H}_0(\xi)$ for $\forall\,\xi\in\mathbb{R}^N$, we can infer from \eqref{pf-eq:nu-P}--\eqref{pf-eq:hH-nu-Q} that $$\left<\nabla u(P(t)),\tau^+\right>=-H(\nabla u(P(t)))$$
 and $$\left<\nabla u(Q(t)),\tau^{-}\right>=-H(\nabla u(Q(t))).$$
Hence, we get
\begin{align}
1=\frac{d}{dt}u(P(t))\Big\vert_{t=t_1}&=\left<\nabla u(P(t_1)),\tau^+\right>\left[\hat{H}_0(x_1)+R'(t_1)\right]\notag\\
&=-H(\nabla u(P(t_1)))\left[\hat{H}_0(x_1)+R'(t_1)\right]\label{pf-eq:duP}
\end{align}
and 
\begin{align}
	1=\frac{d}{dt}u(Q(t))\Big\vert_{t=t_1}&=\left<\nabla u(Q(t_1)),\tau^{-}\right>\left[-H_0(x_1)+R'(t_1)\right]\notag\\
	&=-H(\nabla u(Q(t_1)))\left[-H_0(x_1)+R'(t_1)\right].\label{pf-eq:duQ}
\end{align}
Noting $H(\nabla u(P(t_1)))=H(\nabla u(Q(t_1)))\neq0$, we conclude from \eqref{pf-eq:duP} and \eqref{pf-eq:duQ} that
$$H_0(x_1)+\hat{H}_0(x_1)=0,$$
which contradicts $x_1\neq0$. Consequently, we have $$x(t)\equiv x_0\quad\text{for } t<t_0.$$

Hence, we have shown that given $x\in\mathbb{R}^N$, there exists a unique $t=t(x)\leq t_0$ such that $u(x)=t$ and $\hat{H}_0(x-x_0)=R(t)$. To conclude \eqref{in-eq:u}, for any $t\leq t_0$ we rewrite \eqref{pf-eq:R} as
\begin{equation}\label{pf-eq:re-R}
	e^t=\frac{N\left(\frac{N^2}{N-1}\right)^{N-1}\lambda^N}{\left[1+\lambda^{\frac{N}{N-1}}R^{\frac{N}{N-1}}(t)\right]^N}
\end{equation}
where $$\lambda:=\left[\frac{e^{t_0}}{N\left(\frac{N^2}{N-1}\right)^{N-1}}\right]^{\frac1N}.$$
Letting $t=t(x)$ in \eqref{pf-eq:re-R}, we deduce that
$$	e^{u(x)}=\frac{N\left(\frac{N^2}{N-1}\right)^{N-1}\lambda^N}{\left[1+\lambda^{\frac{N}{N-1}}\hat{H}_0^{\frac{N}{N-1}}(x-x_0)\right]^N}$$
for any $x\in\mathbb{R}^N$.
The proof is thus complete.
\end{proof}

\section{Singularity of solutions to quasi-linear equations of $N$-Laplace type}\label{sec:sing}

The main goal of this section is to establish a result concerning singularity estimates of nonnegative solutions at infinity for a class of $N$-Laplace type equations defined in an exterior domain; see Theorem \ref{sing-thm:log-A-f}. In particular, we derive precise logarithmic asymptotics at infinity or at the origin for nonnegative solutions of the anisotropic $N$-Laplace equation 
\begin{equation}\label{sing-eq:Homo-laplace}
	\Delta_N^H\,v=0
\end{equation} 
in $\mathbb{R}^N\setminus\overline{B_1}$ or in $B_1\setminus\{0\}$, referring to Theorem \ref{sing-thm:lim-log} and Remark \ref{sing-rk:lim-log-O}. In addition, we present a rigidity result for equation \eqref{sing-eq:Homo-laplace} in $\mathbb{R}^N\setminus\{0\}$; see Proposition \ref{sing-prop:rigid}. These are inspired by those obtained in Serrin \cite{Serrin1964,Serrin1965-1,Serrin1965} and in Kichenassamy--Veron \cite{KV1986}.

First, let us consider the following quasi-linear equation 
\begin{equation}\label{sing-eq:div-A-f}
	\mathrm{div}\,\mathcal{A}(\nabla u)=f(x)\quad\text{in }\mathbb{R}^N\setminus\overline{B_1}\,.
\end{equation} 
Here, $\mathcal{A}$ is a vector function on $\mathbb{R}^N$ such that
\begin{equation}\label{sing-eq:condition-A}
	|\mathcal{A}(p)|\le a\,|p|^{N-1}\quad\text{and}\quad\left<\mathcal{A}(p),p\right>\ge b\,|p|^N\qquad\forall p\in\mathbb{R}^N,
\end{equation}
for some constants $a>0$ and $b>0$; $f\in L^1(\mathbb{R}^N\setminus\overline{B_1})$ is a scalar function having the decay estimate
\begin{equation}\label{sing-eq:decay-f}
	0\le f(x)\le\frac{c}{|x|^N} \quad\text{a.e. }x\in\mathbb{R}^N\setminus\overline{B_1},
\end{equation}
for some constant $c>0$. According to \cite{Serrin1964,DiBen1983,Tolksdorf1984}, a solution $u\in W_{\mathrm{loc}}^{1,N}(\mathbb{R}^N\setminus\overline{B_1})$ of \eqref{sing-eq:div-A-f} is of class $C_{\mathrm{loc}}^{1,\alpha}$. We shall push Serrin's singularity estimates \cite{Serrin1965} for \eqref{sing-eq:div-A-f} with $f\equiv0$ into the inhomogeneous setting \eqref{sing-eq:decay-f}. Precisely, we have


\begin{theorem}\label{sing-thm:log-A-f}
	Let $u$ be a nonnegative solution of \eqref{sing-eq:div-A-f}, and assume that \eqref{sing-eq:condition-A} and \eqref{sing-eq:decay-f} hold. If $u(x)\to+\infty$ as $|x|\to+\infty$, then there exists a constant $d>1$ such that  
	\begin{equation}\label{sing-eq:log-sing}
		\frac{1}{d}\log|x|\le u(x)\le d\log|x|
	\end{equation} 
holds in the neighborhood of infinity.
\end{theorem}

To prove this, we present two lemmas. In the sequel, we denote by $\mathscr{Theta}$ the set of Lipschitz continuous functions on $\mathbb{R}^N\setminus B_1$ which are identically one in a neighborhood of infinity but vanish in a neighborhood of $\partial B_1$.

\begin{lemma}\label{sing-lem:A-f-K}
	Let $u$ be as in Theorem \ref{sing-thm:log-A-f}. There is a constant $K>0$ such that
	$$\int_{\mathbb{R}^N\setminus B_1}\left<\mathcal{A}(\nabla u),\nabla\Theta\right>dx+\int_{\mathbb{R}^N\setminus B_1} f\Theta\,dx=K$$
	for any $\Theta\in\mathscr{Theta}$.
\end{lemma}

\begin{proof}
	For any $\Theta_1,\Theta_2\in\mathscr{Theta}$, we can take $\Theta_1-\Theta_2$ as a test function in the weak formulation of \eqref{sing-eq:div-A-f} and obtain accordingly that
	\begin{align*}
		&\int_{\mathbb{R}^N\setminus B_1}\left<\mathcal{A}(\nabla u),\nabla\Theta_1\right>dx+\int_{\mathbb{R}^N\setminus B_1} f\Theta_1\,dx\\
		=&\int_{\mathbb{R}^N\setminus B_1}\left<\mathcal{A}(\nabla u),\nabla\Theta_2\right>dx+\int_{\mathbb{R}^N\setminus B_1} f\Theta_2\,dx,
	\end{align*}
	which implies the existence of $K$.
	
	In order to show $K>0$, let us assume $u(x)<0$ on $|x|=2$ (otherwise we consider $u-k$ for some constant $k$). For $|x|>2$, let
	\begin{equation*}
		\hat{u}(x)=
			\begin{cases}
			0&\text{if }u(x)\le 0\\
			u&\text{if }0<u(x)<1\\
			1&\text{if }u(x)\ge1,
		\end{cases}
	\end{equation*}
and set $\hat{u}(x)\equiv 0$ for $|x|\le2$. Since $u(x)\to+\infty$ as $|x|\to+\infty$, $\hat{u}\in\mathscr{Theta}$. Thus, it follows from \eqref{sing-eq:condition-A} that
\begin{align*}
	K&=\int_{\mathbb{R}^N\setminus B_1}\left<\mathcal{A}(\nabla u),\nabla\hat{u}\right>dx+\int_{\mathbb{R}^N\setminus B_1} f\hat{u}\,dx\\
	&\ge\int_{\mathbb{R}^N\setminus B_1}\left<\mathcal{A}(\nabla\hat{u}),\nabla\hat{u}\right>dx\\
	&\ge b\int_{\mathbb{R}^N\setminus B_1} |\nabla\hat{u}|^N\,dx.
\end{align*}
 Hence, one has $K>0$.
\end{proof}

\begin{remark}
	If $u\in C^1(\mathbb{R}^N\setminus B_1)$, then $K$ can be interpreted as 
   \begin{equation*}
	K=\int_{\partial B_1}\left<\mathcal{A}(\nabla u),\nu\right>d\mathcal{H}^{N-1}+\int_{\mathbb{R}^N\setminus B_1} f\,dx.
	\end{equation*}
This can be validated by taking a sequence $\Theta_j\in\mathscr{Theta}$ such that $\Theta_j\to\chi_{\mathbb{R}^N\setminus B_1}$ in the $L^1$ sense and $\nabla\Theta_j\to\nu\mathcal{H}^{N-1}\llcorner\partial B_1$ in the sense of measures, where $\nu$ is the outer normal to $B_1$. 

In particular, assume that $u=\sup_{\mathbb{R}^N}U-U$ where $U$ is a solution of Liouville equation \eqref{in-eq:Lio}. As done in \eqref{asym-eq:Lio-dual-u}, $u$ is a nonnegative solution of 
\begin{equation*}
	\Delta_N^{\hat{H}} \,u=e^U \quad\text{in }\mathbb{R}^N,
\end{equation*}
which corresponds to \eqref{sing-eq:div-A-f} with $$\mathcal{A}(p)=\hat{H}^{N-1}(p)\nabla\hat{H}(p)\quad\text{and}\quad f(x)=e^{U(x)}\qquad\forall\,x,p\in\mathbb{R}^N.$$
In this case, we obtain by the divergence theorem that
$$K=\int_{\mathbb{R}^N}e^U\,dx.$$
\end{remark}

\begin{lemma}\label{sing-lem:Harnack}
	Let $u$ be as in Theorem \ref{sing-thm:log-A-f}. Then there exist constants $C>0$ and $\sigma_1>0$ such that
	\begin{equation}\label{sing-eq:Harnack-u}
		\max_{|x|=\sigma}u\le C\min_{|x|=\sigma}u
	\end{equation}
for any $\sigma\ge\sigma_1$.
\end{lemma}

\begin{proof}
	Let $u_\sigma(x)=u(\sigma x)$ for $\sigma>2$. It follows from \eqref{sing-eq:div-A-f} that
	\begin{equation}\label{sing-eq:div-A-sigma}
		\mathrm{div}\,\mathcal{A}_\sigma(\nabla u_\sigma)=f_\sigma(x)\quad\text{in }\mathbb{R}^N\setminus\overline{B_\frac12}\,,
	\end{equation}
where $f_\sigma(x)=\sigma^Nf(\sigma x)$ and
\begin{equation*}
	\mathcal{A}_\sigma(p)=\sigma^{N-1}\mathcal{A}(\frac{p}{\sigma})\quad\text{for }p\in\mathbb{R}^N.
\end{equation*}
Given $0<\epsilon<N$ and $y\in\mathbb{S}^{N-1}$, by the Harnack inequality in \cite[Theorem 6]{Serrin1964} applied to the case of equation \eqref{sing-eq:div-A-sigma} we can obtain that
$$\max_{B_{\frac14}(y)}u_\sigma\le C\min_{B_{\frac14}(y)}u_\sigma+\|f_\sigma\|_{L^{\frac{N}{N-\epsilon}}(B_2\setminus B_{\frac12})}^{\frac{1}{N-1}}\,,$$
for some constant $C=C(N,a,\epsilon)>0$ where $a$ is as in \eqref{sing-eq:condition-A}. In view of \eqref{sing-eq:decay-f}, we easily see $$\|f_\sigma\|_{L^{\frac{N}{N-\epsilon}}(B_2\setminus B_{\frac12})}\le C=C(N).$$
Thus, we conclude that
\begin{equation*}
	\max_{|x|=1}u_\sigma\le C_1\min_{|x|=1}u_\sigma+C_2
\end{equation*}
for some positive constants $C_1$ and $C_2$ that are independent of $\sigma$. Moreover, since $u(x)\to+\infty$ as $|x|\to+\infty$, there is $\sigma_1$ such that $\min_{|x|=\sigma}u\ge C_2$ for any $\sigma\ge\sigma_1$. This proves \eqref{sing-eq:Harnack-u}.
\end{proof}

\begin{proof}[Proof of Theorem \ref{sing-thm:log-A-f}] By subtracting a suitable constant, we assume without loss of generality that $u(x)<0$ on $|x|=2$. For $\sigma>2$, let $m(\sigma)=\min_{|x|=\sigma}u$. Since $u(x)\to+\infty$ as $|x|\to+\infty$, there is $\sigma_0>2$ such that $m(\sigma)>0$ for any $\sigma\ge\sigma_0$. We shall first prove for each $\sigma\ge\sigma_0$, 
\begin{equation}\label{sing-eq:min-log}
	m(\sigma)\le C\log\frac{\sigma}{2}
\end{equation}
for some constant $C>0$ independent of $\sigma$.

Given $\sigma\ge\sigma_0$, let
\begin{equation*}
	\bar{u}(x)=
	\begin{cases}
		0&\text{if }u(x)\le 0\\
		u(x)&\text{if }0<u(x)<m(\sigma)\\
		m(\sigma)&\text{if }u(x)\ge m(\sigma)
	\end{cases}
\end{equation*}
for $2\le|x|\le\sigma$, and let $\bar{u}(x)\equiv 0$ for $|x|\le2$ and $\bar{u}(x)\equiv m(\sigma)$ for $|x|\ge\sigma$. Clearly, $\frac{\bar{u}}{m(\sigma)}\in\mathscr{Theta}$. Thus, by Lemma \ref{sing-lem:A-f-K} and \eqref{sing-eq:condition-A}, we have
	\begin{align}
		m(\sigma)K&=\int_{\mathbb{R}^N\setminus B_1}\left<\mathcal{A}(\nabla u),\nabla\bar{u}\right>dx+\int_{\mathbb{R}^N\setminus B_1} f\bar{u}\,dx\notag\\
		&\ge b\int_{\mathbb{R}^N} |\nabla\bar{u}|^N\,dx\,. \label{sing-eq:lower-m}
	\end{align}
We claim 
\begin{equation}\label{sing-eq:capacity}
	\int_{\mathbb{R}^N} |\nabla\bar{u}|^N\,dx\ge \mathcal{H}^{N-1}(\mathbb{S}^{N-1})(m(\sigma))^N\left(\log\frac{\sigma}{2}\right)^{1-N}.
\end{equation}
Indeed, let us consider minimizing the functional
\begin{equation*}
	\int_{\mathbb{R}^N} |\nabla \psi|^N\,dx
\end{equation*}
over all Lipschitz continuous functions $\psi(x)$ that vanish for $|x|\le 2$ and equal identically $m(\sigma)$ for $|x|\ge\sigma$. Observe that this variational problem admits a unique minimizer $\psi_0$ satisfying
\begin{equation*}
	\begin{cases}
		\Delta_N\psi_0=0&\text{in }B_\sigma\setminus\overline{B_2}\\
		\psi_0=0&\text{on }\partial B_2\\
		\psi_0=m(\sigma)&\text{on }\partial B_\sigma.
	\end{cases}
\end{equation*}
Actually, it is seen that $$\psi_0(x)=m(\sigma)\frac{\log|x|-\log2}{\log\sigma-\log2}\quad\text{for }2\le|x|\le\sigma.$$
Hence, one deduces that
$$\int_{\mathbb{R}^N} |\nabla\bar{u}|^N\,dx\ge\int_{\mathbb{R}^N} |\nabla\psi_0|^N\,dx=\mathcal{H}^{N-1}(\mathbb{S}^{N-1})(m(\sigma))^N\left(\log\frac{\sigma}{2}\right)^{1-N},$$
which verifies the claim. Now, combining  \eqref{sing-eq:lower-m} and \eqref{sing-eq:capacity} gives \eqref{sing-eq:min-log}.

Next, let $\sigma_0$ be sufficiently large such that $\int_{\mathbb{R}^N\setminus B_\sigma}f\le\frac{K}{2}$ for any $\sigma\ge\sigma_0$, where $K>0$ is given in Lemma \ref{sing-lem:A-f-K}. Let $M(\sigma)=\max_{|x|=\sigma}u$. We shall show there is a constant $\tilde{C}=\tilde{C}(N,K,a,b)$ such that
\begin{equation}\label{sing-eq:max-log}
	M(\sigma)\ge \tilde{C}\log\frac{\sigma}{\sigma_0}
\end{equation}
for any $\sigma>\sigma_0$.

Given $\sigma>\sigma_0$, set
\begin{equation*}
	\theta=\theta(x,\sigma)=
	\begin{cases}
		0&\text{if }|x|\le\sigma_0\\
		\frac{\log|x|-\log\sigma_0}{\log\sigma-\log\sigma_0}&\text{if }\sigma_0<|x|<\sigma\\
		1&\text{if }|x|\ge\sigma.
	\end{cases}
\end{equation*}
It follows from Lemma \ref{sing-lem:A-f-K} and H\"older inequality that
\begin{align*}
	K&=\int_{\mathbb{R}^N\setminus B_1}\left<\mathcal{A}(\nabla u),\nabla\theta\right>dx+\int_{\mathbb{R}^N\setminus B_1} f\theta\,dx\\
	&\le\left(\int_{B_\sigma\setminus B_{\sigma_0}}|\nabla\theta|^N\,dx\right)^{\frac{1}{N}}\left(\int_{B_\sigma\setminus B_{\sigma_0}}|\mathcal{A}(\nabla u)|^{\frac{N}{N-1}}\,dx\right)^{\frac{N-1}{N}}+\frac{K}{2}.
\end{align*}
	This entails
	\begin{equation}\label{sing-eq:K-A}
		\bar{C}K^{\frac{N}{N-1}}\log\frac{\sigma}{\sigma_0}\le\int_{B_\sigma\setminus B_{\sigma_0}}|\mathcal{A}(\nabla u)|^{\frac{N}{N-1}}\,dx
	\end{equation}
for some constant $\bar{C}=\bar{C}(N)>0$. Define
\begin{equation*}
	V=V(x,\sigma)=
	\begin{cases}
		0&\text{if }|x|<2\\
		\max\{0,u(x)\}&\text{if }2\le|x|\le\sigma\\
		\min\{M(\sigma),u(x)\}&\text{if }|x|>\sigma.
	\end{cases}
\end{equation*}
 Notice that $V=u$ on $B_\sigma\setminus B_{\sigma_0}$ and $V\equiv M(\sigma)$ near infinity (which implies $\frac{V}{M(\sigma)}\in\mathscr{Theta}$). Thus, by Lemma \ref{sing-lem:A-f-K} and \eqref{sing-eq:condition-A} we have
 \begin{align*}
 	KM(\sigma)&=\int_{\mathbb{R}^N\setminus B_1}\left<\mathcal{A}(\nabla u),\nabla V\right>dx+\int_{\mathbb{R}^N\setminus B_1} fV\,dx\\
 	&\ge\int_{\{V=u\}}\left<\mathcal{A}(\nabla u),\nabla u\right>dx\\
 	&\ge b\int_{\{V=u\}}|\nabla u|^N\,dx\\
 	&\ge b\int_{B_\sigma\setminus B_{\sigma_0}}|\nabla u|^N\,dx\\
 	&\ge\frac{b}{a}\int_{B_\sigma\setminus B_{\sigma_0}}|\mathcal{A}(\nabla u)|^{\frac{N}{N-1}}\,dx.
 \end{align*}
This together with \eqref{sing-eq:K-A} yields that $$M(\sigma)\ge\frac{b}{a}\bar{C}K^{\frac{1}{N-1}}\log\frac{\sigma}{\sigma_0},$$
thus proving \eqref{sing-eq:max-log}.

Now, with \eqref{sing-eq:min-log} and \eqref{sing-eq:max-log} in hand, we deduce by Lemma \ref{sing-lem:Harnack} that, for any $\sigma\ge\sigma_0$,
$$C_1\log\frac{\sigma}{\sigma_0}\le m(\sigma)\le M(\sigma)\le C_2\log\frac{\sigma}{2}$$
where $C_1$ and $C_2$ are some positive constants independent of $\sigma$. This shows that \eqref{sing-eq:log-sing} holds, which finishes the proof.
\end{proof}

For the anisotropic $N$-Laplace equation \eqref{sing-eq:Homo-laplace}, the estimates provided in \eqref{sing-eq:log-sing} can be improved to be a sharp version, as stated in the following.

\begin{theorem}\label{sing-thm:lim-log}
	Let $v$ be a nonnegative solution of \begin{equation}\label{sing-eq:Homo}
		\Delta_N^H\,v=0\quad\text{in }\mathbb{R}^N\setminus\overline{B_1}.
	\end{equation}
	Then there exist constants $\gamma\ge0$ and $\beta\ge0$ such that
	\begin{subequations}\label{sing-eq:lim-log}
		\begin{align}
			\lim_{|x|\to+\infty}\left|v(x)-\gamma\log H_0(x)\right|&=\beta\label{sing-eq:lim-log-a}\\
			\lim_{|x|\to+\infty} |x|\left|\nabla(v-\gamma\log H_0(x))\right|&=0.\label{sing-eq:lim-log-b}
		\end{align}
	\end{subequations}
\end{theorem}

In order to prove Theorem \ref{sing-thm:lim-log}, we need the following rigidity result.
\begin{lemma}\label{sing-lem:rigid}
	Assume that $G(x)\in C^1(\mathbb{R}^N\setminus\{0\})\cap L^\infty(\mathbb{R}^N)$ satisfies
	\begin{equation*}
		\Delta_N^H\,(\gamma\log H_0(x)+G(x))=0\quad\text{or }\quad \Delta_N^H\,(-\gamma\log\hat{H}_0(x)+G(x))=0
	\end{equation*}
in $\mathbb{R}^N\setminus\{0\}$, for some constant $\gamma\ge0$. Then $G(x)$ is a constant function.
\end{lemma}

\begin{proof}
	In the Euclidean case ($H=|\cdot|$), this was proved in \cite[Lemma 4.3]{Esposito2018}. By adapting the argument there to the anisotropic case, we showed Lemma \ref{sing-lem:rigid} in our previous paper \cite{CL2022}. More precisely, one may consult \cite[Lemma 3.2]{CL2022} with $\Sigma$ therein taken particularly as $\mathbb{R}^N$ and with $\gamma$ therein corrected to be a nonnegative constant.
\end{proof}

\begin{proof}[Proof of Theorem \ref{sing-thm:lim-log}]
	
	Let us first consider the case that $v$ is bounded on $\mathbb{R}^N\setminus B_2$, that is, there is $C>0$ s.t. $v(x)\le C$ for any $|x|\ge2$. We claim that \eqref{sing-eq:lim-log} holds with $\gamma=0$ in this case.
	
	Indeed, for $R>1$, we let $v_R(x)=v(Rx)$. Notice that $\Delta_N^H\,v_R=0$ and that $v_R$ is uniformly bounded in each compact set of $\mathbb{R}^N\setminus\{0\}$. By the regularity results in \cite{DiBen1983,Tolksdorf1984}, $v_R$ is also uniformly bounded in $C^{1,\alpha}_{\mathrm{loc}}(\mathbb{R}^N\setminus\{0\})$. Thus, it follows from Ascoli--Arzel\`a theorem and a diagonal process that $v_{R_j}\to\tilde{v}$ in $C_{\mathrm{loc}}^1(\mathbb{R}^N\setminus\{0\})$ for a subsequence $R_j\to+\infty$. Notice that $\tilde{v}\in C^1(\mathbb{R}^N\setminus\{0\})\cap L^\infty(\mathbb{R}^N)$ satisfies $\Delta_N^H\tilde{v}=0$ in $\mathbb{R}^N\setminus\{0\}$. Hence, applying Lemma \ref{sing-lem:rigid} with $\gamma=0$ to $\tilde{v}$ yields $\tilde v\equiv\beta$ for some constant $\beta\ge 0$. This implies $\sup_{|x|=R_j} |v(x)-\beta|\to 0$ as $R_j\to+\infty$ and then the maximum principle implies that it holds for the whole sequence $|x|\to+\infty$, thus showing \eqref{sing-eq:lim-log-a} with $\gamma=0$. For \eqref{sing-eq:lim-log-b}, it follows by noting that $\nabla v_R\to 0$ uniformly on $\mathbb{S}^{N-1}$ for any sequence $R\to+\infty$ up to extracting a subsequence.

	Next, let us turn to the case that $v$ is unbounded, i.e., there is a sequence of points $x_j\to\infty$ such that $v(x_j)\to +\infty$. By Harnack inequality \cite[Theorem 6]{Serrin1964} applied to $v$, one has $$\max_{|x|=|x_j|}v\le C\min_{|x|=|x_j|}v\,$$
	for some constant $C=C(N)>0$, from which $\min_{|x|=|x_j|} v\to+\infty$ follows. Then, by using the maximum principle for $v$ on the annular domains $B_{|x_{j+1}|}\setminus B_{|x_j|}$, we get $v(x)\to+\infty$ as $|x|\to+\infty$. Thus, by Theorem \ref{sing-thm:log-A-f} we see that \eqref{sing-eq:log-sing} holds for $v$.
	
	 Let $V_R(x)=\frac{v(Rx)}{\log R}$ for $R>1$. It is seen $\Delta_N^H\,V_R=0$ and $V_R$ is uniformly bounded in $L_{\mathrm{loc}}^\infty(\mathbb{R}^N\setminus\{0\})$ by \eqref{sing-eq:log-sing}. Thus, as in the bounded case, an approximation argument based on \cite{DiBen1983,Tolksdorf1984} and then an application of Lemma \ref{sing-lem:rigid} can lead us to obtain that, along a subsequence $R_j\to+\infty$, $V_{R_j}\to\gamma$ in $C_{\mathrm{loc}}^1(\mathbb{R}^N\setminus\{0\})$ for some constant $\gamma>0$. This implies that $\sup_{\partial B^{H_0}_{R_j}}|\frac{u(x)}{\log R_j}-\gamma|\to0$. Fix $R_0$ such that $\overline{B_1}\subset B^{H_0}_{R_0}$. Then for $0<\epsilon\ll1$, we set
	 	\begin{align*}
	 		\bar{v}_\epsilon(x)&=(\gamma+\epsilon)\log H_0(x)-(\gamma+\epsilon)\log R_0+\sup_{\partial B^{H_0}_{R_0}}v,\\
	 		\underline{v}_\epsilon(x)&=(\gamma-\epsilon)\log H_0(x)-(\gamma-\epsilon)\log R_0+\inf_{\partial B^{H_0}_{R_0}}v.
	 	\end{align*}
Observe that $\Delta_N^H\,\underline{v}_\epsilon=\Delta_N^H\,\bar{v}_\epsilon=0$ and that $\underline{v}_\epsilon\le v\le\bar{v}_\epsilon$ on $\partial B^{H_0}_{R_0}\cup\partial B^{H_0}_{R_j}$ for all sufficiently large $R_j$. By the comparison principle, we infer that $$\underline{v}_\epsilon\le v\le\bar{v}_\epsilon \quad\text{in }\mathbb{R}^N\setminus B^{H_0}_{R_0}.$$ Thus, by letting $\epsilon\to0$ we obtain 
\begin{equation}\label{sing-eq:bdd-v}
	v(x)-\gamma\log H_0(x)\in L^\infty(\mathbb{R}^N\setminus B^{H_0}_{R_0}).
\end{equation}

Now, for $m>1$, we introduce the function 
\begin{equation}\label{sing-eq:scaling-v}
	v_m(y):=v(my)-\gamma\log m.
\end{equation}
By setting $G(x):=v(x)-\gamma\log H_0(x)$, we see $$v_m(y)=\gamma\log H_0(y)+G(my).$$ Since $G(x)\in L^\infty(\mathbb{R}^N\setminus B_{R_0}^{H_0})$ by \eqref{sing-eq:bdd-v}, $v_m(y)$ is bounded in every compact subset of $\mathbb{R}^N\setminus\{0\}$, uniformly in $m$. As argued above, according to \cite{DiBen1983,Tolksdorf1984} one has that $v_m$ is uniformly bounded in $C^{1,\alpha}_{\mathrm{loc}}(\mathbb{R}^N\setminus\{0\})$, and thus, up to a subsequence, one obtains $v_{m}\to v_\infty$ in $C_{\mathrm{loc}}^1(\mathbb{R}^N\setminus\{0\})$ for some function $v_\infty$ satisfying $\Delta_N^H\,v_\infty=0$ in $\mathbb{R}^N\setminus\{0\}$. If we set $$G_\infty(y):=v_\infty(y)-\gamma\log H_0(y),$$ then $G_\infty(y)\in L^\infty(\mathbb{R}^N)$. Hence, by Lemma \ref{sing-lem:rigid} we find $G_\infty\equiv\beta$ for some constant $\beta\in\mathbb{R}$. This implies that, up to a subsequence, $$\lim_{m\to\infty}\left(v(my)-\gamma\log H_0(my)\right)=\beta$$ in the $C^1_{\mathrm{loc}}$ topology. Via the comparison principle, we easily get \eqref{sing-eq:lim-log-a}. Then, since $$\sup_{\partial B^{H_0}_1}|\nabla_y \left(v(my)-\gamma\log H_0(my)\right)-\nabla G_\infty|\to 0$$ for any sequence $m\to+\infty$ up to extracting a subsequence, we deduce that $$\sup_{\partial B^{H_0}_m}H_0(x)\left|\nabla_x\left(v(x)-\gamma\log H_0(x)\right)\right|\to0$$
for $m\to+\infty$, thus proving \eqref{sing-eq:lim-log-b} and completing the proof.
\end{proof}

\begin{remark}\label{sing-rk:lim-log-O}
	The reasoning lines of proof of Theorem \ref{sing-thm:lim-log} can be actually refined as follows. Thanks to Theorem \ref{sing-thm:log-A-f}, it is first seen that a nonnegative solution $v$ of \eqref{sing-eq:Homo} fulfills 
	\begin{equation}\label{sing-eq:upper-log-v}
		v(x)\le\lambda\log|x|+\mu
	\end{equation}
	on $\mathbb{R}^N\setminus B_2$, for some constants $\lambda$ and $\mu$. Then, by exploiting the scaling $\frac{v(Rx)}{\log R}$, it is deduced that
	\begin{equation}\label{sing-eq:sharp-log-v}
		v(x)-\gamma\log H_0(x)\in L^\infty(\mathbb{R}^N\setminus B_2)
	\end{equation}
for some constant $\gamma>0$. Finally, a scaling argument exploiting \eqref{sing-eq:scaling-v} leads to asymptotics \eqref{sing-eq:lim-log}.
	
	Such an approach works also in deriving asymptotic behavior at the origin for nonnegative solutions of equation \eqref{sing-eq:Homo-laplace} defined in the punctured ball $B_1\setminus\{0\}$, as discussed in Kichenassamy--Veron \cite{KV1986} for the case of $N$-harmonic functions. Indeed, we can adapt the proof of Theorem \ref{sing-thm:lim-log} to obtain that any nonnegative solution $v$ of \eqref{sing-eq:Homo-laplace} in $B_1\setminus\{0\}$ fulfills \eqref{sing-eq:upper-log-v} near the origin, according to Serrin \cite[Theorem 12]{Serrin1964}. Then by adapting the procedure arguing \eqref{sing-eq:sharp-log-v} and concluding \eqref{sing-eq:lim-log} to the punctured case, we are able to validate $$v(x)+\gamma\log\hat{H}_0(x)\in L^\infty(B_\frac12)$$
	 and further deduce that in this case, \eqref{sing-eq:lim-log-a} and \eqref{sing-eq:lim-log-b} still hold with $|x|\to+\infty$ replaced by $|x|\to0$ and with $-\gamma\log H_0$ replaced by $\gamma\log\hat{H}_0$.
\end{remark}

In terms of Remark \ref{sing-rk:lim-log-O}, we are able to derive a rigidity result in Proposition \ref{sing-prop:rigid} below, which improves Lemma \ref{sing-lem:rigid}. Notice that, in the Euclidean case, this result as well as its Corollary \ref{sing-cor:rigid} to be stated has been obtained in  \cite[Theorem 2.2]{KV1986}. Here we extend it to the anisotropic setting by adapting the idea there.

\begin{proposition}\label{sing-prop:rigid}
	Let $v$ be a solution of \begin{equation}\label{sing-eq:Homo-pun}
		\Delta_N^H\,v=0\quad\text{in }\mathbb{R}^N\setminus\{0\}
	\end{equation}
and assume that 
\begin{equation}\label{sing-eq:v-lam-mu}
	|v(x)|\le\lambda\left|\log|x|\right|+\mu\quad\text{in } \mathbb{R}^N\setminus\{0\}\,,
\end{equation}
for some constants $\lambda\ge0$ and $\mu\ge0$. Then $$v(x)=\gamma\log H_0(x)+\beta\quad\text{or}\quad v(x)=-\gamma\log\hat{H}_0(x)+\beta $$
for some constant $\gamma\ge0$ and $\beta\in\mathbb{R}$.
\end{proposition}

\begin{proof}
	Under assumption \eqref{sing-eq:v-lam-mu}, we see that the function $v_r(x):=\frac{v(rx)}{\log r}$ with $r>0$ is bounded in $L_{\mathrm{loc}}^\infty(\mathbb{R}^N\setminus\{0\})$, uniformly in $r$. From this, we can argue as in the proof of Theorem \ref{sing-thm:lim-log} to obtain first that, for some constant $\gamma_1$, $v_{r}\to\gamma_1$ in $C_{\mathrm{loc}}^1(\mathbb{R}^N\setminus\{0\})$ as $r\to+\infty$. If $\gamma_1\ge0$, as argued for \eqref{sing-eq:bdd-v}, we have
	\begin{equation}\label{sing-eq:v-gamma1}
		\|v-\gamma_1\log H_0\|_{L^\infty(\mathbb{R}^N\setminus B_1)}\le C
	\end{equation}
for some constant $C>0$; moreover, we can infer that 
\begin{equation}\label{sing-eq:lim-v-gamma-beta}
	\lim_{|x|\to+\infty}\left(v(x)-\gamma_1\log H_0(x)\right)=\beta
\end{equation}
for some constant $\beta$. If $\gamma_1<0$, by modifying the function $\log H_0$ as $\log\hat{H}_0$ in the above argument, we can see that \eqref{sing-eq:v-gamma1} and \eqref{sing-eq:lim-v-gamma-beta} hold with $H_0$ replaced by $\hat{H}_0$.

On the other hand, as explained in Remark \ref{sing-rk:lim-log-O}, one can argue in the same way to find that for some constant $\gamma_2$, $v_r\to\gamma_2$ in $C_{\mathrm{loc}}^1(\mathbb{R}^N\setminus\{0\})$ as $r\to0$, and then to obtain that
\begin{equation}\label{sing-eq:v-gamma2}
	\|v-\gamma_2\log H_0\|_{L^\infty(B_1)}\le C\quad\text{if }\gamma_2\ge 0,
\end{equation}
or that $$\|v-\gamma_2\log\hat{H}_0\|_{L^\infty(B_1)}\le C\quad\text{if }\gamma_2<0.$$
Moreover, in this case, \eqref{sing-eq:lim-log} holds for $v$ with $|x|\to+\infty$ replaced by $|x|\to0$ (and with $H_0$ replaced by $\hat{H}_0$) if $\gamma_2\ge 0$ ($\gamma_2<0$). In particular, we have, for $\gamma_2\ge 0$,
$$\nabla v(x)=\gamma_2\nabla(\log H_0(x))+o(|x|^{-1})\quad\text{as }|x|\to0$$
(for $\gamma_2<0$, this holds with $H_0$ replaced by $\hat{H}_0$), which implies $\nabla v\in L^q(B_1)$ for any $1\le q<N$. Then, given $R>0$, a standard approximation procedure based on the divergence theorem shows that
\begin{equation}\label{sing-eq:Eq-v}
	-\Delta_N^H\,v=C_{\gamma_2}\,\delta_0=
	\begin{cases}
		-\Delta_N^H\,(\gamma_2\log H_0) &\text{if }\gamma_2\ge 0\\
		-\Delta_N^H\,(\gamma_2\log\hat{H}_0) &\text{if }\gamma_2<0,
	\end{cases}
\end{equation}
 in the sense of distributions in $B_R$, where $\delta_0$ is the Dirac measure at the origin and
 \begin{equation*}
 	C_{\gamma_2}=
 	\begin{cases}
 		-\gamma_2^{N-1}N\mathcal{L}^{N}(B_1^{H_0})&\text{if }\gamma_2\ge 0\\
 		(-\gamma_2)^{N-1}N\mathcal{L}^{N}(B_1^{\hat{H}_0})&\text{if }\gamma_2<0.
 	\end{cases}
 \end{equation*}

If $\gamma_1=\gamma_2=0$, i.e., $v$ is globally bounded, we directly conclude from Lemma \ref{sing-lem:rigid} that $v\equiv\beta$ for some constant $\beta$. Next, let us assume that $\gamma_1\neq 0$ or $\gamma_2\neq 0$; we will prove $\gamma_1=\gamma_2:=\gamma$ in this situation.

Indeed, if $\gamma_1>0$, we introduce 
\begin{gather*}
	\bar{v}_{\epsilon,r}(x)=(\gamma_1+\epsilon)(\log H_0(x)-\log r)+\gamma_2\log r+C,\\
	\underline{v}_{\epsilon,r}(x)=(\gamma_1-\epsilon)(\log H_0(x)-\log r)+\gamma_2\log r-C,
\end{gather*}
for given $0<\epsilon\ll 1$ and $0<r<\mu_1$ where $\mu_1$ is as in \eqref{asym-eq:bound-H0}. Notice from \eqref{sing-eq:v-gamma1} that 
\begin{equation}\label{sing-eq:lim-v-gamma2}
	\lim_{|x|\to+\infty}\frac{v(x)-\gamma_2\log r\pm C}{\log H_0(x)-\log r}=\gamma_1.
\end{equation} 
This, together with \eqref{sing-eq:v-gamma2}, entails that, 
$$\underline{v}_{\epsilon,r}\le v\le\bar{v}_{\epsilon,r} \quad\text{on }\partial B^{H_0}_{r}\cup\partial B^{H_0}_{R}$$ for all large $R$. Thus, by the comparison principle and letting $\epsilon\to0$, we obtain
$$\gamma_1\log H_0(x)+(\gamma_2-\gamma_1)\log r-C\le v(x)\le\gamma_1\log H_0(x)+(\gamma_2-\gamma_1)\log r+C$$
for any $x$ with $H_0(x)\ge r$. Fixing $x\neq 0$ and letting $r\to0$ in this inequality implies $\gamma_1=\gamma_2$. 

Likewise, if $\gamma_1<0$, we can modify $\log H_0$ above as $\log\hat{H}_0$ to conclude $\gamma_1=\gamma_2$.

If $\gamma_2>0$, let us modify the above $\bar{v}_{\epsilon,r}$ and $ \underline{v}_{\epsilon,r}$ to
\begin{gather*}
	\bar{v}_{\epsilon,r}(x)=(\gamma_2+\epsilon)(\log H_0(x)-\log r)+\gamma_1\log r+C,\\
	\underline{v}_{\epsilon,r}(x)=(\gamma_2-\epsilon)(\log H_0(x)-\log r)+\gamma_1\log r-C,
\end{gather*}
for given $0<\epsilon\ll 1$ and $r>\mu_2$ where $\mu_2$ is as in \eqref{asym-eq:bound-H0}. By \eqref{sing-eq:v-gamma2}, we see that \eqref{sing-eq:lim-v-gamma2} holds with $|x|\to+\infty$ replaced by $|x|\to0$ and with an exchange of $\gamma_1$ and $\gamma_2$, which, together with \eqref{sing-eq:v-gamma1}, implies
$$\underline{v}_{\epsilon,r}\le v\le\bar{v}_{\epsilon,r} \quad\text{on }\partial B^{H_0}_{r}\text{ and as }|x|\to0.$$ 
Thus, similarly as in the case $\gamma_1>0$, the comparison principle and a procedure of letting $\epsilon\to0$ and then $r\to+\infty$ forces $\gamma_1=\gamma_2$. This is also true for $\gamma_2<0$ after a modification of $H_0$ to $\hat{H}_0$. 

Now, for $0<\varepsilon\ll 1$, by \eqref{sing-eq:lim-v-gamma-beta} with $\gamma_1=\gamma>0$ there is $T=T(\varepsilon)$ such that $$\gamma\log H_0(x)+\beta-\varepsilon\le v(x)\le\gamma\log H_0(x)+\beta+\varepsilon\quad\text{for }|x|>T.$$
By virtue of \eqref{sing-eq:Eq-v}, we can deduce from the comparison principle that $$\gamma\log H_0(x)+\beta-\varepsilon\le v(x)\le\gamma\log H_0(x)+\beta+\varepsilon\quad\text{in } B_R$$
for any $R>T$. Letting $\varepsilon\to0$ yields $$v=\gamma\log H_0(x)+\beta\quad\text{in }\mathbb{R}^N\setminus\{0\}.$$
Due to the same reason, we can obtain 
$$v=\gamma\log\hat{H}_0(x)+\beta\quad\text{in }\mathbb{R}^N\setminus\{0\}$$
when $\gamma<0$.

The proof is completed.
\end{proof}

\begin{corollary}\label{sing-cor:rigid}
	Any nonnegative solution of \eqref{sing-eq:Homo-pun} is a constant.
\end{corollary}

\begin{proof}
	According to \cite[Theorem 12]{Serrin1964} and Theorem \ref{sing-thm:log-A-f}, a nonnegative solution of \eqref{sing-eq:Homo-pun} fulfills \eqref{sing-eq:v-lam-mu}. Thus, the assertion follows from Proposition \ref{sing-prop:rigid}.
\end{proof}

 To conclude, we remark that it is possible to generalize the logarithmic estimates in Theorem \ref{sing-thm:log-A-f} to a more general class of equations of $N$-Laplace type as considered in Serrin \cite{Serrin1964,Serrin1965-1} where singularity analysis at a finite point was performed for a very wide class of equations. More in general, we believe that the results in this section can be extended to a setting of $p$-Laplace type with $p\neq N$. These are in progress and will be presented in another paper.

\medskip

\subsection*{Acknowledgements}
The authors have been supported by the Research Project of the Italian Ministry of University and Research (MUR) Prin 2022 ``Partial differential equations and related geometric-functional inequalities'', grant number 20229M52AS\_004. The first author has been also partially supported by the ``Gruppo Nazionale per l'Analisi Matematica, la Probabilit\`a e le loro Applicazioni'' (GNAMPA) of the ``Istituto Nazionale di Alta Matematica'' (INdAM, Italy).

\end{document}